\documentclass[12pt]{amsart}

\usepackage{amsfonts}
\usepackage{amsmath}
\usepackage{amssymb}
\usepackage{amsthm}
\usepackage{bm}
\usepackage{enumerate} 
\usepackage[margin=1.2in]{geometry}
\usepackage{tikz-cd}
\usepackage{hyperref}
\usepackage{mathrsfs}

\usepackage{xcolor}

\theoremstyle{plain}
\newtheorem{theorem}{Theorem}[section]
\newtheorem{lemma}[theorem]{Lemma}
\newtheorem{proposition}[theorem]{Proposition}
\newtheorem{corollary}[theorem]{Corollary}

\theoremstyle{definition}
\newtheorem{definition}[theorem]{Definition}

\newtheorem{remark}[theorem]{Remark}

\newtheorem{example}[theorem]{Example}

\numberwithin{equation}{section}

\newcommand\N{\mathbb{N}}
\newcommand\Z{\mathbb{Z}}
\newcommand\R{\mathbb{R}}
\newcommand\C{\mathbb{C}}

\newcommand\D{\mathcal{D}}

\renewcommand\S{\mathcal{S}}
\newcommand\K{\mathcal{K}}
\renewcommand\O{\mathcal{O}}
\newcommand\OC{\O_{C}}
\newcommand\OM{\O_{M}}
\newcommand\B{\mathcal{B}}

\newcommand\V{\mathcal{V}}
\newcommand\W{\mathcal{W}}

\newcommand\ev[2]{\langle#1,#2\rangle}

\DeclareMathOperator\supp{supp}

\DeclareMathOperator\loc{loc}

\DeclareMathOperator\condwM{wM}

\DeclareMathOperator\condN{N}
\DeclareMathOperator\condwN{wN}

\DeclareMathOperator\semloc{mc}
\DeclareMathOperator\approxim{ess}

\begin{document}

\title{Weighted function spaces: convolutors, multipliers, and mollifiers}

\author[L. Neyt]{Lenny Neyt}
\address{University of Vienna\\ Faculty of Mathematics\\ Oskar-Morgenstern-Platz 1 \\ 1090 Wien\\ Austria}
\thanks{The research of L. Neyt was funded in whole by the Austrian Science Fund (FWF) 10.55776/ESP8128624. For open access purposes, the author has applied a CC BY public copyright license to any author-accepted manuscript version arising from this submission.}
\email{lenny.neyt@univie.ac.at}

\author[Y. Sawano]{Yoshihiro Sawano}
\thanks{Y. Sawano was supported by Grant-in-Aid for Scientific Research (C) (19K03546), the Japan Society for the Promotion of Science.}
\address{Department of Mathematics \\ Chuo University \\ Kasuga 1-13-27, Bunkyo \\ Tokyo 112-8551 \\ Japan}
\email{sawano@math.chuo-u.ac.jp}

\begin{abstract}
We study smooth function spaces of Gelfand-Shilov type, with global behavior governed through a translation-invariant Banach function space and localized via a weight function system. We clarify the roles of the translation-invariant Banach function space, convolution, and pointwise multiplication in connection with the weight function system. Our primary goal is to characterize these function spaces---as well as the corresponding convolutor and multiplier spaces---through mollification. For this purpose, we introduce the moment-wise decomposition factorization property for pairs of compactly supported smooth functions, and establish complete characterizations in terms of mollifications with these windows.
\end{abstract}

\maketitle

\section{Introduction}
\label{s1}

The specific nature of a generalized function can be investigated through various techniques. Well-known approaches include Littlewood--Paley theory, atomic decompositions, and wavelet decompositions; see \cite{T-ThFuncSp, T-ThFuncSpII} for comprehensive overviews.

An alternative method involves the use of mollifiers. Let \( \D = \D(\R^d) \) denote the space of compactly supported smooth functions on \( \R^d \), and let \( \D' \) be its dual, the space of distributions. For any \( \phi \in \D \), we define the scaled function
\begin{equation}
	\label{eq:Mollifier} 
    \phi_j = 2^{jd} \phi(2^j \cdot), \qquad j \in \N.
\end{equation}
Then, for any \( f \in \D' \), we have
\begin{equation} \label{eq:250421-12}
    \lim_{j \to \infty} f * \phi_j = \left( \int \phi \right) f
\end{equation}
in the sense of distributions. Each convolution \( f * \phi_j \) is a smooth function. 

This naturally raises the question: can certain properties of $f$---such as regularity or decay---be inferred from the sequence $(f * \phi_j)$? Addressing this question lies at the core of the sequential approach to distribution theory \cite{A-M-S-ThDistSeqApproach, K-DistAppliedMath}.
Problems of this kind also arise naturally in the study of (singular) partial differential equations and are closely related to the nonlinear theory of generalized functions, as discussed in \cite{O-MultDistApplPDE}. Characterizing local regularity via mollifiers is a classical theme in generalized function theory (see, for instance, \cite{H-HoldZygRegAlgGenFunc, P-S-V-RegPropDistSeq}). In the present work, however, our focus is on the \emph{global} behavior of functions. More specifically, we aim to determine when a distribution belongs to certain Gelfand--Shilov-type spaces---spaces of smooth functions adhering to specified decay---or to their associated convolutor and multiplier spaces.

Let \( E \) be a solid translation-invariant Banach function space \cite{F-G-BanachSpIntGroupRepAtomicDecomp} (see Definition~\ref{defi:250421-1} below). A standard example is \( E = L^p(\R^d) \) for \( p \in [1, \infty] \).
Given a sequence \( \W = (w_n)_{n \in \N} \) of continuous functions on \( \R^d \) such that
\[ 1 \leq w_n \leq w_{n + 1}, \qquad n \in \N, \]
we define
\[
    E^\infty_\W = \left\{ f \in \D^\prime \mid f^{(\alpha)} w_n \in E \text{ for all } n \in \N \text{ and } \alpha \in \N^d \right\}.
\]
For \( E = L^p(\R^d) \), \( p \in [1, \infty] \), the spaces \( \mathcal{K}^p_\W = (L^p(\R^d))^\infty_\W \) were introduced by Gelfand and Shilovin \cite{G-S-GenFunc2, G-S-GenFunc3} for the study of differential equations. 
These spaces provide a unifying framework encompassing classical test function spaces such as the Schwartz space \cite{S-ThDist}, the Hasumi--Silva space \cite{H-NoteNDimTempUltraDist, H-SP-ThGenFunc}, and the spaces \( \mathcal{D}_{L^p} \), \( p \in [1, \infty) \), and $\mathcal{B}$, where a function \( f \in \mathcal{D}_{L^p} \), respectively $f \in \mathcal{B}$, exactly when \( f \) and all its derivatives belong to \( L^p(\R^d) \), respectively $L^\infty(\R^d)$.
A substantial portion of this note will be dedicated to a detailed study of the space \( E^\infty_\W \), as well as its associated convolutor and multiplier spaces, under minimal assumptions on \( E \) and \( \W \).

The space \( E \), particularly its norm, governs how global behavior is interpreted within \( E^\infty_\W \). The sequence of weights \( \W \) then acts to localize this behavior.
In the literature, the analysis of many function spaces relies on $E$ having a two-module structure:
\begin{itemize}
    \item \( E \) forms an \( L^1 \)-module under convolution;
    \item \( E \) forms an \( L^\infty \)-module under pointwise multiplication. 
\end{itemize}
Inspired by the notion of homogenous Banach spaces in the work of Katznelson \cite{K-IntroHarmAnal} on $\R$, the introduction of such a type of Banach function spaces goes back to the work of Feichtinger in \cite{F-MultBanachSpFuncGroups} (see also \cite{B-F-BanachSpDistTwoMod}) in abstract harmonic analysis.
Notably, they were used as building blocks for the development of coorbit theory \cite{F-G-BanachSpIntGroupRepAtomicDecomp}.
We refer to \cite{F-UpcomingTerminology} and the references therein for a thorough overview of Banach function spaces and their applications.
The two-module structure enables us to treat numerous results 
on the algebraic structure of function spaces in a unified manner, as illustrated in \cite{D-V-WeighIndLimitsSpUltaDiffFunc, D-V-TopPropConvSpSTFT, D-HarmonicAnalCommutativeBanachSp, D-TISubspLp, W-ConvDLp}.
We refer to \cite{D-N-SeqSpRepTMIB, D-N-V-InclRelGSSp, D-P-V-NewDistSpTIB, FeGu21a, FeGu23b, N-ConvTMIB} for more recent applications of this perspective, often in the context of solid translation-invariant Banach spaces of bounded type, or other closely related Banach function spaces.
Our setting will be formally introduced and explored in Section~\ref{s2}.

We define weight function systems \( \W = (w_n)_{n \in \N} \) in Section~\ref{s3} and investigate several conditions imposed on them. 
For our analysis to work, we will require that \( \W \) is 
\emph{weak-moderate $\mathrm{(\condwM)}$}, that is,
\[
(\condwM) \quad
    \text{for all } n \in \N, \text{ there exists } m \ge n \text{ such that } \sup_{|y| \leq 1} \frac{w_n(\cdot + y)}{w_m(\cdot)} = O(1).
\]
Assuming \(\mathrm{(\condwM)}\) ensures that the space \( E^\infty_\W \) is translation-invariant and closed under convolution with elements in \( \D \) (see Proposition~\ref{p:EWConvCompact} and use \( E^\infty_\W \subset E_\W \)).
Additional conditions on \( \W \) are introduced later in Section \ref{s3}
to further localize \( E^\infty_\W \) and to control the behavior of its associated convolutor and multiplier spaces.

A thorough investigation of the Gelfand--Shilov-type space \( E^\infty_\W \) and its convolutor and multiplier spaces is given in Sections~\ref{sec:SpEinftyW} and~\ref{sec:ConvMultSp}.
If the weight functions in \( \W \) exhibit relative decay, that is, when for each $n \in \N$ there is an integer $m \geq n$ such that $w_n / w_m$ has a certain form of vanishing behavior at infinity, the space becomes more localized, and the global behavior in the definition of \( E^\infty_\W \) plays a diminished role.
Most notably, if \( \W \) satisfies the condition \(\mathrm{(\condN)}\) (already used in \cite{G-S-GenFunc3} to characterize nuclearity), that is, for each \( n \in \N \) there exists \( m \geq n \) such that \( w_n / w_m \in L^1(\R^d) \), then \( E^\infty_\W \) becomes independent of the specific Banach function space \( E \) used to define it; see Theorem~\ref{t:IndependenceTIBF}.

Analogous results are established for the associated convolutor and multiplier spaces; see in particular Theorems~\ref{t:ExtensionOC} and~\ref{t:CharOMSingleW}.
Our analysis is also optimal: in each case, we prove that the assumptions imposed on \( \W \) (under the standing condition \(\mathrm{(\condwM)}\)) are necessary.

In the final section, Section~\ref{s6}, we characterize 
the aforementioned spaces by means of mollification.
Our objective is twofold: to minimize the number of mollifying windows used and to relax their structural requirements as much as possible.

Generally, using a single window tends to be rigid.
When the window has zero mean, the setting essentially corresponds to a continuous wavelet transform \cite{D-TenLectWavelets, H-Wavelets}.
However, proper wavelet analysis necessitates windows with non-compact support, thereby tying the analysis more directly to the localization imposed by \( \W \).
Such a situation will be examined in a forthcoming work.

If the window has a non-zero mean, one can still characterize spaces
for certain instances of \( E^\infty_\W \) (see Proposition~\ref{p:250408-1}). 
However, it appears difficult in this setting to achieve effective cancellation of low-frequency background components.
The situation changes when using two windows.
In Definition \ref{def:MDFP}, we introduce the \emph{moment-wise decomposition factorization property}
(MDFP for short) for a pair of windows $(\chi^0, \chi^1) \in \D \times \D$.
By design, excluding some additional technical conditions, these functions have the special property that for any $L \in \N$, there exist $\varphi^\ell, \psi^\ell \in \D$, $\ell \in \{0, 1\}$, such that the moments of $\psi^0, \psi^1$ up to order $L$ vanish, and (recall our notation \eqref{eq:Mollifier})
	\begin{equation}
		\label{eq:IntroDecomp}
		f = \sum_{\ell = 0}^1 \left[ \{f * \chi^\ell_0\} * \varphi^\ell_0 + \sum_{j = 0}^\infty 4^{-j} \Delta (\{f * \chi^\ell_j\} * \psi^\ell_j) \right] 
	\end{equation} for  every distribution $f$, where $\Delta$ denotes the Laplace operator.
Exploiting the fact that the vanishing order $L$---which governs the control of erratic low-frequency behavior---can be chosen freely, we extrapolate the behavior of $f$ from its mollifications $f * \chi^0_j$ and $f * \chi^1_j$.
Based on this idea,
we characterize the space $E^\infty_\W$ and its corresponding convolutor and multiplier spaces; 
see Theorems~\ref{t:KWDecomp}, \ref{t:OCDecomp}, and \ref{t:OMDecomp}.
To demonstrate the existence of a pair of windows satisfying the moment-wise decomposition factorization property (MDFP), we employ tools from the theory of ultradifferentiable functions in the sense of Braun-Meise-Taylor~\cite{B-M-T-UltradiffFuncFourierAnal}.

To conclude this introduction, we characterize spaces defined via a single weight in terms of mollifiers.
Let $\omega : \R^d \to [0, \infty)$ be a non-decreasing, continuous function.
For $p \in [1, \infty]$, we define the function space
	\[ \K^p_\omega := \{ f \in C^\infty(\R^d) \mid f^{(\alpha)} e^{n \omega(\cdot)} \in L^p(\R^d) \text{ for all } n \in \N \text{ and } \alpha \in \N^d \} , \]
the convolutor space
	\[ \OC^\prime(\D, \K^p_\omega) := \{ f \in \D^\prime \mid f * \phi \in \K^p_\omega \text{ for all } \phi \in \D \} , \]
and the multiplier space
	\[ \OM(\K^\infty_\omega, \K^p_\omega) := \{ f \in C^\infty(\R^d) \mid f \cdot \phi \in \K^p_\omega \text{ for all } \phi \in \K^\infty_\omega \} . \]
When $\omega \equiv 0$, then $\K^p_{1} = \D_{L^p}$, $p \in [1, \infty)$, and $\K^\infty_1 = \mathcal{B}$. Here $\OC^\prime(\D, \D_{L^p}) = \D^\prime_{L^p}$, the dual of $\D_{L^q}$ with $q$ the conjugate index of $p$, and $\OM^\prime(\B, \D_{L^p}) = \D_{L^p}$, see \cite{S-ThDist}.
We note that if $\log (1 + |t|) = O(\omega(t))$ as $|t| \to \infty$, the definition of $\K^p_\omega$ is independent of $p \in [1, \infty]$, in view of Theorem \ref{t:IndependenceTIBF}.
In the case where $\omega(t) = \log(1 + |t|)$, then $\K^p_{\log(1 + |\cdot|)}$ is exactly the Schwartz space $\mathcal{S}$, and $\OC^\prime(\D, \S) = \OC^\prime$ and $\OM(\S, \S) = \OM$ the convolutor and multiplier spaces, respectively, as introduced by Schwartz in \cite{S-ThDist}. 
Another notable case is when $\omega(t) = |t|$, where $\K^p_{|\cdot|}$ becomes the Hasumi--Silva space \cite{H-NoteNDimTempUltraDist, H-SP-ThGenFunc}.
The spaces
$\K^p_\omega$,
$\OC^\prime(\D, \K^p_\omega)$,
and
$\OM(\K^\infty_\omega, \K^p_\omega)$
are characterized as follows:

\begin{theorem}
	\label{t:Intro}
	Let $p \in [1, \infty]$.
	Let $\omega : \R^d \to [0, \infty)$ be a non-decreasing, continuous function such that $\omega$ is moderate, i.e.,
	\begin{equation}\label{eq:moderateness} \sup_{|y| \leq 1} \frac{\omega(x + y)}{\omega(x)} < \infty , \qquad x \in \R^d . \end{equation}
	Let $(\chi^0, \chi^1) \in \D \times \D$ satisfy the MDFP.
	For any $f \in \D^\prime$, we have that:
		\begin{itemize}
			\item[$(i)$] $f \in \K^p_\omega$ if and only if
there exists $r>0$ such that for all $n \in {\mathbb N}$ and $\alpha \in {\mathbb N}^d$
we have
				\[ \sup_{j \in \N, \, \ell \in \{0, 1\}} 2^{-rj} \| \partial^\alpha [ f * \chi^\ell_j ] e^{n \omega} \|_{L^p} < \infty . \]
			\item[$(ii)$] $f \in \OC^\prime(\D, \K^p_\omega)$ if and only if
for all $n \in {\mathbb N}$ there exists $r>0$ such that
				\[ \sup_{j \in \N, \, \ell \in \{0, 1\}} 2^{-rj} \| [ f * \chi^\ell_j ] e^{n \omega} \|_{L^p} < \infty . \]
			\item[$(iii)$] $f \in \OM(\K^\infty_\omega, \K^p_\omega)$ if and only if there exists $r>0$ such that for all $\alpha \in {\mathbb N}^d$,
we can find $n \in {\mathbb N}$ satisfying
				\[ 
\sup_{j \in \N, \, \ell \in \{0, 1\}} 2^{-rj} \| \partial^\alpha [ f * \chi^\ell_j ] e^{- n \omega} \|_{L^p} < \infty . \]
		\end{itemize}
\end{theorem}
The proof of Theorem \ref{t:Intro}  will be given at the end of Section \ref{s6} as an application of our general results.

\section{Translation-invariant Banach function spaces}
\label{s2}

We consider function spaces that are invariant and bounded under translation, which will serve as the building blocks of our theory. These spaces govern the global behavior of the functions we study. 

The translation of a function $f$ by a vector $x \in \R^d$ is defined as
\[
  T_{x}f(t) := f(t - x),
\]
and the reflection of $f$ about the origin is defined as
\[
  \check{f}(t) := f(-t).
\]

For any set $A \subseteq \R^d$, we write $1_A$ to denote the indicator function of $A$. A compact subset $K$ of $\R^d$ is written as $K \Subset \R^d$. For any $x \in \R^d$ and $R > 0$, we write $\overline{B}(x, R)$ for the closed ball of radius $R$ centered at $x$.

By \( L^1_{\mathrm{loc}}(\mathbb{R}^d) \), we denote the space of all locally integrable functions on \( \mathbb{R}^d \). 
For any compact set \( K \Subset \mathbb{R}^d \), we write \( C_{\mathrm{c}}(K) \) and \( \mathcal{D}(K) \) for the Banach space of continuous functions and the Fr\'echet space of smooth functions, respectively, whose supports are contained in \( K \). 
We equip
\[
C_{\mathrm{c}}(\mathbb{R}^d) := \bigcup_{K \Subset \mathbb{R}^d} C_{\mathrm{c}}(K)
\quad \text{and} \quad
\mathcal{D} := \bigcup_{K \Subset \mathbb{R}^d} \mathcal{D}(K)
\]
with their canonical strict \((LF)\)-space topologies. 
Moreover, for each \( L \in \mathbb{N} \), we define the following closed subspace of \( \mathcal{D}(K) \):
\[
\mathcal{D}_L(K)
    := \left\{ \varphi \in \mathcal{D}(K) \;:\;
        \int_{\mathbb{R}^d} x^{\alpha} \varphi(x) \, dx = 0
        \text{ for all } |\alpha| < L \right\}.
\]
\begin{definition}\label{defi:250421-1}
A Banach space $(E, \|\cdot\|_E)$ is called a \emph{translation-invariant Banach function space} 
(\emph{TIBF} for short) \emph{of bounded type} 
on ${\mathbb R}^d$ if $E$ is non-trivial, the continuous inclusion $E \subseteq L^{1}_{\loc}({\mathbb R}^d)$ holds, and $E$ satisfies the following three conditions:
\begin{itemize}
  \item[(A.1)] $T_{x} E \subseteq E$ for all $x \in {\mathbb R}^{n}$.
  \item[(A.2)] There exists $C_0 > 0$ such that $\|T_{x} f\|_{E} \leq C_0 \|f\|_{E}$ for all $x \in {\mathbb R}^{d}$ and $f \in E$.
  \item[(A.3)] $E$ is a Banach module of $L^1(\R^d)$ with respect to convolution.
\end{itemize}
The Banach space $E$ is called \emph{solid} if for all $f \in E$ and $g \in L^{1}_{\loc}({\mathbb R}^d)$ we have that
\[
  |g(x)| \leq |f(x)| \text{ for almost all } x \in {\mathbb R}^d \quad \Longrightarrow \quad g \in E \text{ and } \|g\|_{E} \leq \|f\|_{E}.
\] 
	\end{definition}

	\begin{remark}
    Let $(E, \|\cdot\|_E)$ be a solid TIBF of bounded type.
    \begin{itemize}
\item[$(i)$]  Every element of $L^\infty({\mathbb R}^d)$ with compact support belongs to $E$ (cf.\ \cite[Lemma 3.9]{F-G-BanachSpIntGroupRepAtomicDecomp}).
In particular, $1_K \in E$ for every compact 
subset $K \subseteq {\mathbb R}^{n}$.
\item[$(ii)$] By $(i)$ we have that $C_{\rm c}(\mathbb{R}^d) \subseteq E$. Since $C_{\rm c}(\mathbb{R}^d)$ may be endowed with the topology of an inductive limit of Banach spaces, $L^1_{\loc}(\R^d)$ with that of a Fr\'echet space, and since $C_{\rm c}(\mathbb{R}^d) \subseteq L^1_{\loc}(\mathbb{R}^d)$ continuously, it follows from the closed graph theorem that $C_{\rm c}(\mathbb{R}^d) \subseteq E$ continuously.
\end{itemize}
	\end{remark}
	
	\begin{remark}
		\label{rem:StandardSituation}
In the terminology of \cite[Definition~3.1]{B-F-BanachSpDistTwoMod}, 
a solid TIBF \( (E, \|\cdot\|_E) \) is in a \emph{standard situation} 
with respect to the Banach algebra \( C_0(\mathbb{R}^d) \),
the Banach space of all continuous functions vanishing at infinity.
In fact, \( E \) is a two Banach module: 
it is a module over \( L^1(\mathbb{R}^d) \) with respect to convolution 
and a module over \( L^\infty(\mathbb{R}^d) \) with respect to pointwise multiplication. 
Hence, for any \( f \in E \), \( \varphi \in L^1(\mathbb{R}^d) \), and \( \psi \in L^\infty(\mathbb{R}^d) \),
\begin{equation}
\label{eq:ModuleBounds}
    \| f * \varphi \|_E \le \|f\|_E \|\varphi\|_{L^1},
    \qquad 
    \| f \cdot \psi \|_E \le \|f\|_E \|\psi\|_{L^\infty}.
\end{equation}
This property will play a key role in the subsequent analysis.
	\end{remark}

For the remainder of this work, 
as a standing assumption,
\( (E, \|\cdot\|_E) \), 
or simply \( E \) for short, will always denote a solid TIBF of bounded type.

\subsection{Examples of translation-invariant Banach function spaces}
\label{s2.1}

We present several examples of solid TIBF of bounded type.
	
\begin{example}
\label{ex:BasicTIBF}
Rearangement invariant Banach function spaces (see \cite{L-T-ClassicBanachSpII}) encompass many well-known examples of solid TIBF of bounded type.

$\bullet$ \textbf{Lebesgue spaces:} Let $p \in [1,\infty]$. The Lebesgue space $L^p = L^p({\mathbb R}^d)$ is a solid TIBF of bounded type on ${\mathbb R}^d$. 
We also define 
	\[ L^0 = L^0({\mathbb R}^d) := \{ f \in L^\infty \mid \text{for all } \varepsilon > 0 \text{ there is } K \Subset \R^d \text{ so that } \|(1 - 1_K) f\|_{L^\infty} < \varepsilon \} . \]
Thus, $L^0$ consists of all $L^\infty$-functions vanishing at infinity, and is a closed subspace of $L^\infty$. 
Then $L^0$ is a solid TIBF of bounded type on ${\mathbb R}^d$.

$\bullet$ \textbf{Lorentz spaces:} Let $1 \leq p \leq q \leq \infty$. The Lorentz space $L^{p, q} = L^{p, q}(\R^d)$ \cite{B-S-InterpolOp} is the space of all $f \in L^1_{\loc}(\R^d)$ for which the norm
	\[ \|f\|_{L^{p, q}} := \begin{cases} \left( \int_0^\infty \left(r^{\frac{1}{p}} f^{*}(r) \right)^{q} \frac{dr}{r} \right)^{\frac{1}{q}} ,  & q < \infty , \\ \sup_{r > 0} r^{\frac{1}{p}} f^{*}(r) , & q = \infty , ~ \end{cases}  \]
is finite, where
	\[ f^*(r) = \inf \{ s > 0 : |\{ x \mid |f(x)| > s\}| \leq r \} . \]
Each $L^{p, q}$ is a solid TIBF of bounded type.

$\bullet$ \textbf{Orlicz spaces:} Let $\Phi : [0, \infty) \to [0, \infty)$ be a \emph{Young function}, that is, a convex, continuous function increasing to infinity such that $\Phi(0) = 0$ and $\Phi(r) > 0$ for $r > 0$.
The Orlicz space $L^\Phi = L^\Phi(\R^d)$ \cite{R-R-ThOrliczSp} is the space of all $f \in L^1_{\loc}(\R^d)$ for which the norm
	\[ \|f\|_{L^\Phi} := \inf \left\{ R > 0 \mid \int_{\R^d} \Phi\left(\frac{|f(x)|}{R}\right) dx < \infty \right\} . \]
is finite. Then $L^\Phi$ is a solid TIBF of bounded type.
%
%
%
\end{example}

Another family of well-known function spaces is the Morrey spaces.

\begin{example}[Morrey spaces]
Let \(1 \le q \le p < \infty\). For a function \(f \in L^q_{\mathrm{loc}}(\mathbb{R}^d)\), its \emph{Morrey norm} is defined by
\[
\| f \|_{\mathcal{M}^{p}_q}
:=
\sup_{x \in \mathbb{R}^d, R>0}
|B(x,R)|^{\frac{1}{p}-\frac{1}{q}}
\left(
\int_{B(x,R)}|f(y)|^{q}\,\mathrm{d}y
\right)^{\frac{1}{q}},
\]
where \(B(x,R)\) denotes the open ball centered at \(x \in \mathbb{R}^d\) with radius \(R>0\).  
The \emph{Morrey space} \(\mathcal{M}^{p}_q = \mathcal{M}^{p}_q(\mathbb{R}^d)\) consists of all \(f \in L^q_{\mathrm{loc}}(\mathbb{R}^d)\) such that \(\| f \|_{\mathcal{M}^{p}_q} < \infty\). 
It is a solid TIBF of bounded type.
These spaces interpolate between Lebesgue spaces (\(q = p\)) and spaces 
with weaker local control (\(q < p\)), and are useful in regularity theory for PDEs \cite{book}.
\end{example}

We may also associate Wiener-type spaces to any solid TIBF of bounded type.
\begin{example}[Wiener amalgam spaces]
To $(E, \|\cdot\|_E)$ we may associate the Banach space
	\[ E_{\rm seq} = \left\{ c = (c_j)_{j \in \Z^d} \in \C^{\Z^d} \mid \sum_{j \in \Z^d} c_j T_j 1_{[0, 1]^d} \in E \right\}, \]
which we endow with the norm $\|c\|_{E_{\rm seq}} = \|\sum_{j \in \Z^d} |c_j| T_j 1_{[0, 1]^d}\|_E$.
For example, $L^p_{\rm seq} = \ell^p$ for $p \in [1, \infty]$ and $L^0_{\rm seq} = c_0$. 
One has the continuous inclusions $\ell^1 \subseteq E_{\rm seq} \subseteq \ell^\infty$ \cite[Lemma 3.5(b)]{F-G-BanachSpIntGroupRepAtomicDecomp}.

Then, for $p \in [1, \infty]$, we define the Wiener amalgam space \cite{F-BanachConvAlgWienerType,F-S-Amalgams}
	\[ W(L^p, E_{\rm seq}) = \{ f \in L^p_{\loc}(\R^d) \mid (\|f \cdot T_j 1_{[0, 1]^d}\|_{L^p})_{k \in \Z^d} \in E_{\rm seq} \}  \]
with the norm $\|f\|_{W(L^p, E_{\rm seq})} = \| (\|f \cdot T_j 1_{[0, 1]^d}\|_{L^p})_{j \in \Z^d} \|_{E_{\rm seq}}$.
Then $W(L^p, E_{\rm seq})$ is again a solid TIBF of bounded type (see \cite[Corollary 2]{F-BanachConvAlgWienerType}).
Here, $L^p$ can be seen as the ``local'' component, while $E$ measures the ``global'' component of $W(L^p, E_{\rm seq})$.
For example, we have
	\[ W(L^p, \ell^q) = \left\{ f \in L^1_{\loc}(\R^d) \mid \sum_{j \in \Z^d} \left( \int_{[0, 1]^d} |f(x + j)|^{p} dx \right)^{\frac{q}{p}} < \infty \right\} , \]
with the obvious modifications when $p = \infty$ or $q = \infty$.
\end{example}

The space $W(L^1, E_{\rm seq})$ in particular can be seen as the largest solid TIBF of bounded type that convolutes $C_{\rm c}(\R^d)$ into $E$.

\begin{lemma} \label{l:W(L^1, E_d)ConvSp}
We have 
\begin{equation}\label{eq:ConvAmalgam} W(L^1, E_{\rm seq}) = \{ f \in L^1_{\loc}(\R^d) \mid |f| * \phi \in E \text{ for all } \phi \in C_{\rm c}(\R^d) \} . \end{equation}
An equivalent norm for $W(L^1, E_{\rm seq})$ is given by $\| f \|_{W(L^1, E_{\rm seq}), 2} := \| |f| * 1_{[0, 1]^d} \|_E$.
\end{lemma}

\begin{proof}
Feichtinger and Gr\"ochenig showed that
$W(L^1, E_{\rm seq})$ is continuously contained in the right-hand side of 
\eqref{eq:ConvAmalgam};
see \cite[Theorem 7.1]{F-G-BanachSpIntGroupRepAtomicDecompII}.
On the other hand, let $f \in L^1_{\loc}(\R^d)$ now be an element of the right-hand side of \eqref{eq:ConvAmalgam}, and take any non-negative function $\chi \in \D$ such that $\chi \geq 1_{[0, 1]^d}$.
Then $\| f \cdot T_j 1_{[0, 1]^d} \|_{L^1} \leq |f| * \chi (j)$.
As in the proof of \cite[Lemma 2.7]{D-N-V-InclRelGSSp}, we see that then necessarily $(|f| * \chi(j))_{j \in \Z^d} \in E_{\rm seq}$, so we conclude $f \in W(L^1, E_{\rm seq})$.
The final statement about the equivalent norm now follows from the open mapping theorem.
\end{proof}
	

Each solid TIBF of bounded type is located between the Wiener amalgam spaces $W(L^\infty, \ell^1)$ and $W(L^1, \ell^\infty)$.

\begin{lemma}[{\cite[Lemma 3.9]{F-G-BanachSpIntGroupRepAtomicDecomp}}]
\label{l:InclusionWienerSpace}
We have the continuous inclusions
	\[ W(L^\infty, \ell^1) \subseteq E \subseteq W(L^1, \ell^\infty) . \]
\end{lemma}

\subsection{Essential part and module completion}
\label{s2.2}

Following \cite{B-F-BanachSpDistTwoMod}
and Remark \ref{rem:StandardSituation}, to any solid TIBF of bounded type we may associate two others.

\begin{definition}
	Let $(E, \|\cdot\|)$ be a solid TIBF of bounded type.
	We define the following two function spaces.
	
	$(i)$ The essential part $E_{\approxim}$ is given by
		\[ E_{\approxim} := \{ f \in E \,:\, \text{for all } \varepsilon > 0 \text{ there is } K \Subset \R^d \text{ so that } \|(1 - 1_K) f\|_E \leq \varepsilon \} \]
	endowed with the original norm of $E$.
		
	$(ii)$ The multiplicative module completion $E_{\semloc}$ is given by
		\[ E_{\semloc} := \{ f \in L^1_{\loc}(\mathbb{R}^d) \,:\, 1_K f \in E \text{ for all } K \Subset \mathbb{R}^d \text{ and } \|f\|_{E_{\semloc}} := \sup_{K \Subset \R^d} \| 1_{K} f \|_E < \infty \} \]
	endowed with the norm $\|\cdot\|_{E_{\semloc}}$.
\end{definition}

For example $L^0=(L^\infty)_{\approxim}$.
The spaces $E_{\approxim}$ and $E_{\semloc}$ have the following natural descriptions.
In particular, in view of \cite[Lemma 4.3]{B-F-BanachSpDistTwoMod}, $E_{\semloc}$ then coincides with the $C_0(\R^d)$ module completion of $E$.

\begin{proposition}
\label{p:E=MEiffSemiLocal}
The space $E_{\approxim}$ is a closed subspace of $E$.
Moreover
	\begin{equation}
		\label{eq:SemiLocalMultSpC0}
		E_{\semloc} = \{ f \in L^1_{\loc}(\mathbb{R}^d) \,:\, f \cdot \phi \in E \text{ for all } \phi \in C_0(\mathbb{R}^d) \} . 
	\end{equation}
\end{proposition}

\begin{proof}
By definition, $E_{\approxim}$ is the closure in $E$ of the compactly supported elements of $E$, and thus forms a closed subspace of $E$.
It remains to establish \eqref{eq:SemiLocalMultSpC0}. 
Suppose first that \( f \in L^1_{\mathrm{loc}}(\mathbb{R}^d) \) is not contained in \( E_{\mathrm{mc}} \).
There are two possibilities.

First, there exists a compact set \( K \Subset \mathbb{R}^d \) such that \( 1_K f \notin E \).
In this case, take any nonnegative function \( \phi \in C_{\mathrm{c}}(\mathbb{R}^d) \) satisfying \( \phi \equiv 1 \) on \( K \).
Then, by the solidity of \( E \), we have \( f \cdot \phi \notin E \), 
and hence \( f \) does not belong to the right-hand side of \eqref{eq:SemiLocalMultSpC0}.

In the second case, there exists a sequence of compact sets \( (K_n)_{n \in \mathbb{N}} \) such that
\( K_n \subset \mathring{K}_{n+1} \), \( \mathbb{R}^d = \bigcup\limits_{n \in \mathbb{N}} \mathring{K}_n \), and 
\[
    \| 1_{K_{n+1}} f - 1_{K_n} f \|_E \ge n + 2 , \qquad n \in \mathbb{N}.
\]
Now choose a continuous function \( \phi \) such that 
\((n + 2)^{-1/2} \le \phi(x) \le (n + 1)^{-1/2}\) for all \( x \in K_{n+1} \setminus K_n \).
Then \( \phi \in C_0(\mathbb{R}^d) \), and by the solidity of \( E \),
\[
    \| 1_{K_{n+1}} [f \cdot \phi] - 1_{K_n} [f \cdot \phi] \|_E 
    \ge \sqrt{n + 2}, 
    \qquad n \in \mathbb{N}.
\]
Since \( E \) is solid, it follows that \( f \cdot \phi \notin E \).
Consequently, the right-hand side of \eqref{eq:SemiLocalMultSpC0} is contained in \( E_{\mathrm{mc}} \).

\medskip

Conversely, suppose that \( f \in E_{\mathrm{mc}} \).
Take any \( \phi \in C_0(\mathbb{R}^d) \), and let \( (R_n)_{n \in \mathbb{N}} \) be a strictly increasing divergent sequence of positive numbers such that
\(|\phi(x)| \le (n + 1)^{-2}\) whenever \(|x| \ge R_n\).
Since
\[
    |f \cdot \phi| 
    = |1_{\overline{B}(0, R_0)} [f \cdot \phi]| 
      + \sum_{n \in \mathbb{N}}
        |1_{\overline{B}(0, R_{n+1})} [f \cdot \phi]
          - 1_{\overline{B}(0, R_n)} [f \cdot \phi]|
    \quad \text{almost everywhere},
\]
it follows from the solidity of \( E \) that
\begin{align*}
    \| f \cdot \phi \|_E
    &\le \| 1_{\overline{B}(0, R_0)} [f \cdot \phi] \|_E
      + \sum_{n \in \mathbb{N}}
        \| 1_{\overline{B}(0, R_{n+1})} [f \cdot \phi]
          - 1_{\overline{B}(0, R_n)} [f \cdot \phi] \|_E \\
    &\le \left( \|\phi\|_{L^\infty} + \frac{\pi^2}{3} \right)
        \|f\|_{E_{\mathrm{mc}}}.
\end{align*}
This proves the opposite inclusion.
Hence, \eqref{eq:SemiLocalMultSpC0} holds.
\end{proof}

We now consider several properties of $E_{\approxim}$ and $E_{\semloc}$.
The first states that they are again solid TIBFS of bounded type; its proof can be done as in \cite[Theorems 4.2 and 4.6]{B-F-BanachSpDistTwoMod} with minor alterations to show solidity.

\begin{lemma}
	\label{l:ApproximSemilocTIBF}
	Both $(E_{\approxim}, \|\cdot\|_E)$ and $(E_{\semloc}, \|\cdot\|_{E_{\semloc}})$ are also solid TIBF of bounded type.
\end{lemma}

\begin{lemma}
	\label{l:DotandSemilocProp}
	The following statements are true.
	\begin{itemize}
		\item[$(i)$] $E_{\approxim} \subseteq E \subseteq E_{\semloc}$ continuously.
		\item[$(ii)$] $(E_{\approxim})_{\approxim} = E_{\approxim}$ and $(E_{\semloc})_{\semloc} = E_{\semloc}$.
		\item[$(iii)$] $(E_{\approxim})_{\semloc} = E_{\semloc}$ and $(E_{\semloc})_{\approxim} = E_{\approxim}$.
		\item[$(iv)$] For any other solid TIBF of bounded type $(F, \|\cdot\|_F)$,  
$F_{\approxim} = E_{\approxim}$ if and only if $F_{\semloc} = E_{\semloc}$.
	\end{itemize}
\end{lemma}

\begin{proof}
$(i)$ By definition we have that $E_{\approxim} \subseteq E$ continuously. Now, for any $f \in E$ and any $R > 0$, since $E$ is solid, $\|1_{\overline{B}(0, R)} f\|_E \leq \|f\|_E$, which shows that also $E \subseteq E_{\semloc}$ continously. 

$(ii)$ and $(iii)$ Follow by \cite[Theorem 1.1(D)]{B-F-BanachSpDistTwoMod}.

%
%

$(iv)$ If $F_{\approxim} = E_{\approxim}$, then $F_{\semloc} = (F_{\approxim})_{\semloc} = (E_{\approxim})_{\semloc} = E_{\semloc}$
from $(iii)${\rm;} if 
conversely
$F_{\semloc} = E_{\semloc}$, then $F_{\approxim} = (F_{\semloc})_{\approxim} = (E_{\semloc})_{\approxim} = E_{\approxim}$ again from $(iii)$.
\end{proof}

\begin{remark}
By Lemma \ref{l:DotandSemilocProp}, 
the multiplication module completion $E_{\semloc}$ of $E$ is uniquely determined by the essential part $E_{\approxim}$ of $E$, and vice-versa. Different solid TIBFs of bounded type can have the same essential part and multiplication module completion, see Lemma \ref{l:MEBasicTIBF} to follow.
\end{remark}

We review two properties on $E$ to guarantee that it coincides with its essential part, respectively, its multiplicative module completion. These properties stem from the work of Luxemburg and Zaanen, see \cite[Chapter 15]{Z-Integration}.

\begin{lemma}
\label{l:ECoincidesWithDerivSp}
The following statements are true.
\begin{itemize}
\item[$(i)$] $E = E_{\approxim}$ if $E$ has an \emph{abolsutely continuous norm}{\rm:} If $f \in E$ and $(f_n)_{n \in \N}$ is a bounded sequence in $E$ such that $0 \leq f_n \nearrow f$ pointwise, then $f_n \to f$ in $E$.
\item[$(ii)$] $E = E_{\semloc}$ if $E$ has the \emph{Fatou property}: If $f \in L^1_{\loc}(\R^d)$ and $(f_n)_{n \in \N}$ is a bounded sequence in $E$ such that $0 \leq f_n \nearrow f$ pointwise, then $f \in E$ and $\|f\|_E = \lim\limits_{n \in \N} \|f_n\|_E$.
\end{itemize}
In particular, $E = E_{\approxim} = E_{\semloc}$ when $E$ is reflexive.
\end{lemma}

\begin{proof}
$(i)$ can be found in \cite[Proposition 1.4]{F-CompactnessTIB}, while $(ii)$ follows easily by taking the sequence $f_n = f 1_{\overline{B}(0, n)}$.
The last statement is shown in \cite[Chapter 15, Section 73, Theorem 2]{Z-Integration}.
\end{proof}

We end this section by considering the essential parts and multiplication module completions for the examples
that have appeared up until now in this text. 
We first make the ensuing general observation about the Wiener amalgam spaces of the form $W(L^1, E_{\rm seq})$.

\begin{lemma}
\label{l:EssPartModComplAmalgam}
It holds
	\[ W(L^1, E_{\rm seq})_{\approxim} = W(L^1, ((E_{\approxim})_{\rm seq}) \quad \text{and} \quad W(L^1, E_{\rm seq})_{\semloc} = W(L^1, ((E_{\semloc})_{\rm seq}) . \]
\end{lemma}

\begin{proof}
We only show the first identity, the second follows by similar arguments.
Take any $f \in W(L^1, ((E_{\approxim})_{\rm seq})$. 
Let $\phi \in C_{\rm c}(\mathbb{R}^d)$ be such that $0 \leq \phi \leq 1$ and $\phi \equiv 1$ on $[0, 1]^d$.
For any $K \Subset \R^d$ we put $\widetilde{K} = K - [0, 1]^d$.
Then
	\[
		| 1_{[0, 1]^d} * (|f| - 1_{\widetilde{K}} |f|) | 
		\leq | (1 - 1_{K})[1_{[0, 1]^d} * |f|] | 
		\leq | (1 - 1_{K})[|f| * |\phi|] | .
	\]
Note that $|f| * |\phi| \in E_{\approxim}$ by Lemma \ref{l:W(L^1, E_d)ConvSp}.
Then, by the solidity of $E$ and Lemma \ref{l:W(L^1, E_d)ConvSp} once more, we conclude $f \in W(L^1, E_{\rm seq})_{\approxim}$.

Let now $f$ be in $(W(L^1, E_{\rm seq}))_{\approxim}$. 
By Lemma \ref{l:W(L^1, E_d)ConvSp}, for any $\varepsilon > 0$ there is a compact set $K_{0,\varepsilon} \Subset \R^d$ such that
	\[ \| [|f| - 1_{K_{0,\varepsilon}} |f|] * 1_{[0, 1]^d} \|_E \leq \varepsilon . \]
Put $K_{1, \varepsilon} = K_{0, \varepsilon} + [0, 1]^d$. Then, by the solidity of $E$,
	\[ \| \{1 - 1_{\widetilde{K}_{1, \varepsilon}}\} [|f| * 1_{[0, 1]^d}] \|_E \leq \| [|f| - 1_{K_{0, \varepsilon}} |f|] * 1_{[0, 1]^d} \|_E \leq \varepsilon . \]
For any $\phi \in C_{\rm c}(\mathbb{R}^d)$, there exist, independent of $\varepsilon$, constants $C_{1}, \ldots, C_k > 0$ and $x_1, \ldots, x_k \in \R^d$ such that $|\phi| \leq \sum_{j = 1}^k C_j T_{x_j} 1_{[0, 1]^d}$.
Put $K_{2, \varepsilon} = \bigcup_{j=1}^k T_{x_j} K_{1, \varepsilon}$.
We have
\[ \{1 - 1_{K_{2, \varepsilon}}\} (|f| * \phi) \leq \sum_{j = 1}^k C_j T_{x_j} \left[ \{1 - 1_{K_{1, \varepsilon}}\} (|f| * 1_{[-1, 1]^d}) \right] . \]
By the solidity of $E$
	\[ \| \{1 - 1_{K_{2, \varepsilon}}\} (|f| * \phi) \|_E \leq \varepsilon C_0 \sum_{j = 1}^k C_j . \]
Since $\varepsilon$ was chosen arbitrarily, we conclude that $|f| * \phi \in E_{\approxim}$.
Therefore, by Lemma \ref{l:W(L^1, E_d)ConvSp}, we have $f \in W(L^1, ((E_{\approxim})_{\rm seq})$.
\end{proof}

We now list more explicit examples.
Note that $W(L^1, c_0)$ is different from $W(L^1, \ell^\infty)$ (e.g., it does not contain constant functions).

\begin{lemma}
\label{l:MEBasicTIBF}
$(i)$ $L^p = (L^p)_{\approxim} = (L^p)_{\semloc}$ for $p \in [1, \infty)$.\\
$(ii)$ $L^\infty = (L^\infty)_{\semloc} = (L^0)_{\semloc}$\quad and\quad $L^0 = (L^\infty)_{\approxim} = (L^0)_{\approxim}$;\\
$(iii)$ $L^{p, q} = L^{p, q}_{\approxim}$ when $1 \leq p \leq q < \infty$ and $L^{p, q} = L^{p, q}_{\semloc}$ for $1 \leq p \leq q \leq \infty$. \\
$(iv)$ $L^\Phi = L^\Phi_{\approxim} = L^\Phi_{\semloc}$ for any Young function $\Phi$; \\
$(v)$ $W(L^\infty, \ell^1) = (W(L^\infty, \ell^1))_{\approxim} = (W(L^\infty, \ell^1))_{\semloc}$. \\
$(vi)$ $W(L^1, c_0) = (W(L^1, \ell^\infty))_{\approxim} = (W(L^1, c_0))_{\approxim}$.\\
$(vii)$ $W(L^1, \ell^\infty) = (W(L^1, \ell^\infty))_{\semloc} = (W(L^1, c_0))_{\semloc}$.
\end{lemma}

\begin{proof}
$(i)$--$(iv)$ are either clear or follow from Lemma \ref{l:ECoincidesWithDerivSp}.
(Recall that Young functions are finitely valued.)
$(v)$--$(vii)$ are consequences of Lemmas \ref{l:DotandSemilocProp} and \ref{l:EssPartModComplAmalgam}.
\end{proof}

It can happen that $E$ does not coincide with both $E_{\approxim}$ and $E_{\semloc}$.
\begin{remark}
\label{rem:SameApproximSp}
Work in ${\mathbb R}^2$ and
consider $E = W(L^1, c_0(\Z) \widehat{\otimes} \ell^\infty(\Z))$.
By Lemma~\ref{l:EssPartModComplAmalgam}, we have $E_{\approxim} = W(L^1, c_0(\Z^2))$ and $E_{\semloc} = W(L^1, \ell^\infty(\Z^2))$, which are different.
\end{remark}

We find the ensuing improvement of  Lemma \ref{l:InclusionWienerSpace} for the essential parts.

\begin{corollary}
	\label{c:EapinW(L1,C0)}
	We have $E_{\approxim} \subseteq W(L^1, c_0)$ continuously.
\end{corollary}

\begin{proof}
Let
$F$ be
a solid TIBF of bounded type such that $E \subseteq F$. 
Then $E_{\approxim} \subseteq F_{\approxim}$ continuously.
	The result now follows from Lemmas \ref{l:InclusionWienerSpace} and \ref{l:MEBasicTIBF}$(vi)$.
\end{proof}

We conclude this section by considering Morrey spaces.
\begin{example}
Let $1 \le q \le p<\infty$.
It follows from the Fatou property
of
$\mathcal{M}^p_q$
that $(\mathcal{M}^p_q)_{\semloc} = \mathcal{M}^p_q$.
However, if $q < p$, then
$({\mathcal M}^p_q)_{\approxim}$ 
differs from
${\mathcal M}^p_q$, see \cite[Example 136, p.~365]{book}.
\end{example}

\section{Weight Function Systems}
\label{s3}

In this text, a \emph{weight function} $\omega : \mathbb{R}^d \to [1, \infty)$ refers to a locally bounded measurable function. 
A sequence of weight functions $\W = (w_n)_{n \in \mathbb{N}}$ is called a \emph{weight function system} if it is increasing, that is, if $w_n \leq w_{n+1}$ for all $n \in \mathbb{N}$.

We consider the following three conditions on a weight function system $\W$:

\begin{itemize}
    \item[$(\condwM)$] For every $n \in \mathbb{N}$, there exist $m \geq n$ and a constant $C > 0$ such that for all $x \in \mathbb{R}^d$ and $y \in \overline{B}(0, 1)$, we have
    \[
        w_n(x+y) \leq C w_m(x) .
    \]
    
    \item[$(\condwN)$] For every $n \in \mathbb{N}$, there exists $m \geq n$ such that $w_n / w_m$ belongs to $L^0$.
    
    \item[$(\condN)$] For every $n \in \mathbb{N}$, there exists $m \geq n$ such that $w_n / w_m$ belongs to $L^1$.
\end{itemize}

Note that if $\mathcal{W}$ satisfies $(\condwM)$, then, by iteration, the set $\overline{B}(0, 1)$ can be replaced by any compact subset with a non-empty interior.
Moreover, when $\mathcal{W}$ satisfies $(\condwM)$, if $\mathcal{W}$ also satisfies $(\condN)$, it follows that $\mathcal{W}$ satisfies $(\condwN)$ as well (see \cite[Lemma 3.1]{D-N-V-NuclGSSpKernThm} and Lemma \ref{l:SmoothWeightFuncSyst} below).

\begin{definition}\label{defi:2.1}
Let $\V=(v_n)_{n \in {\mathbb N}}$, 
$\W=(w_n)_{n \in {\mathbb N}}$ be weight function systems.
Then $\W$ is called \emph{$\V$-moderate} if, for each $n \in {\mathbb N}$, there exist $m \geq n$ and a constant $C > 0$ such that, for all $x, y \in \mathbb{R}^d$,
\[
    w_n(x + y) \leq C w_m(x) v_m(y).
\]
Any $\V$-moderate sequence
$\W$ automatically satisfies $(\condwM)$.
\end{definition}

\begin{example} 
Let $\omega : \R^d \to [0, \infty)$ be a locally bounded measurable function.
Then
	\[ \W_\omega = (\exp[n \cdot \omega])_{n \in \N} \]
is a weight function system.
It satisfies $(\condwM)$ if $\omega$ is \emph{moderate} \cite{F-Gewichtsfunktionen,F-UpcomingTerminology,FeGu21a,FeGu23b}, i.e., if $\omega$ satisfies \eqref{eq:moderateness}.

We now have the following classical examples
of moderate weights and nonexamples of $(\condwN)$ and $(\condN)$.
Here $\langle x \rangle = (1 + |x|^2)^{1/2}$ for any $x \in \R^d$.
\begin{itemize}
    \item[(i)] \textbf{Constant weight sequence:} 
    Let $\mathbf{1} = \W_{0}$. This sequence is \( \mathbf{1} \)-moderate but fails the condition \((\condwN)\).
    
    \item[(ii)] \textbf{Weight sequence of logarithmic growth:} 
    Define $\W_{\log} = \W_{\log[\log[\langle\cdot\rangle]]}$. This sequence is \( \W_{\log} \)-moderate and satisfies \((\condwN)\), but it does not satisfy \((\condN)\).

    \item[(iii)] \textbf{Weight sequence of polynomial growth:} 
    Define $\W_{\mathrm{pol}} = \W_{\log[\langle\cdot\rangle]}$. This sequence is \( \W_{\mathrm{pol}} \)-moderate and satisfies \((\condN) \).

    \item[(iv)] \textbf{Weight sequence of exponential growth:} 
    Let $\W_{\exp} = \W_{|\cdot|}$. This sequence is \( \W_{\exp} \)-moderate and satisfies \((\condN)\).

\end{itemize}
\end{example}

Two weight function systems $\W$ and $\V$ are called \emph{equivalent} if, for each  $n \in {\mathbb N}$, there exist $m \geq n$ and a constant $C > 0$ such that
\[
    w_n \leq C v_m \quad \text{and} \quad v_n \leq C w_m.
\]
Note that $(\condwM)$, $(\condwN)$, $(\condN)$, and moderateness (relative to another fixed weight function system) remain invariant under equivalence.

In this paper, the symmetry of a set $K$ signifies that $y \in K$ if and only if $-y \in K$.
For any $\eta \in L^1_{\loc}(\mathbb{R}^d)$ and $\chi \in \D$, we write $\eta_\chi := \eta * \chi$.
\begin{lemma}
\label{l:SandwhichSmoothFunc}
Let $K \Subset \mathbb{R}^d$ be symmetric and $\chi \in \D(K)$ be non-negative with $\int \chi = 1$.
Let $\omega, \nu : \mathbb{R}^d \to (0, \infty)$ be locally integrable.
The following two statements are true:
	\begin{itemize}
		\item[$(i)$] If 
\begin{equation}\label{eq:250421-11}
\omega(a + b) \leq C \nu(a) , \qquad a \in \R^d, ~ b \in 2K
=\{y+y' \,:\,y,y' \in K\},
\end{equation} for some $C > 0$, then, $\omega(x + y) \leq C \nu_\chi(x)$ for all $x \in \mathbb{R}^d$ and $y \in K$, and for each $\alpha \in \N^d$ there exists some $C_\alpha > 0$ such that $|\omega^{(\alpha)}_\chi| \leq C_\alpha \nu${\rm;}
		\item[$(ii)$] If $\omega \leq \nu$, then, $\omega_\chi \leq \nu_\chi$.
	\end{itemize}
\end{lemma}

\begin{proof}
$(i)$ 
From (\ref{eq:250421-11})
 	\[ \omega(x + y) = \int_{K} \chi(u) \omega(x + y) du \leq C \int_{\mathbb{R}^d} \chi(u) \nu(x - u) du = C \nu_\chi(x) . \]
Meanwhile, for any $\alpha \in \N^d$ we find
	\[ |\omega_\chi^{(\alpha)}(x)| = \left| \int_{K} \chi^{(\alpha)}(y) \omega(x - y) dx \right| \leq C \|\chi^{(\alpha)}\|_{L^1} \nu(x) . \]
	Thus, $C_\alpha=C \|\chi^{(\alpha)}\|_{L^1}$ does the job.
    
$(ii)$ 
Since $\chi$ is non-negative, this is clear from the definition of the convolution.
\end{proof}
We now see that the condition $(\condwM)$ is quite flexible,
in the sense that it allows us to consider the differentiation of weights.

\begin{lemma}
\label{l:SmoothWeightFuncSyst}
Let $\mathcal{W} = (w_n)_{n \in \mathbb{N}}$ be a weight function system satisfying $(\condwM)$. Then there exists an equivalent weight function system $\mathscr{W} = (\omega_n)_{n \in \mathbb{N}}$ consisting of smooth functions on $\mathbb{R}^d$ such that
\[
c_{n,\alpha}:=    \sup_{x \in \mathbb{R}^d} \frac{|\partial^\alpha \omega_n(x)|}{\omega_m(x)} < \infty
\]
for every $\alpha \in \mathbb{N}^d$
and $n \in \mathbb{N}$ and some $m \in \mathbb{N}$
depending on $\alpha$ and $n$.
\end{lemma}

\begin{proof}
Take any $\chi \in \mathcal{D}(\overline{B}(0, 1/2))$ such that $\chi$ is non-negative and $\int \chi = 1$.
For each $n \in \mathbb{N}$, define $\omega_n := (w_n)_\chi$ as in Lemma~\ref{l:SandwhichSmoothFunc}.

By Lemma~\ref{l:SandwhichSmoothFunc}$(ii)$, each $\omega_n$ is a smooth function and satisfies $1 \leq \omega_n \leq \omega_{n+1}$. 
In particular, $\mathscr{W} = (\omega_n)_{n \in \mathbb{N}}$ is a weight function system consisting of smooth functions.

The equivalence of $\mathscr{W}$ and $\mathcal{W}$, along with the stated derivative estimate, follows directly from the assumption that $\mathcal{W}$ satisfies condition $(\condwM)$ and from Lemma~\ref{l:SandwhichSmoothFunc}$(i)$.
\end{proof}

\section{Spaces of Gelfand-Shilov-type}
\label{sec:SpEinftyW}

In this section, we introduce the general class of Gelfand--Shilov-type spaces that will be studied throughout the remaining part of this paper.
They combine the global behavior coming from a solid TIBF of bounded type with the localization induced by a weight function system.

Let $\omega : \mathbb{R}^d \to (0, \infty)$ be a locally integrable function such that $1 /\omega$ is locally bounded.  
We define the weighted Banach space
\[
E_\omega := \{ f \in L^1_{\loc}(\mathbb{R}^d) \,:\, f \omega \in E \}
\]
endowed with the norm
\[
\|f\|_{E_\omega} := \| f \omega \|_E.
\]

For any $n \in \N$, we define the space
\[
E^n_\omega := \left\{ f \in \D^\prime \,:\, f^{(\alpha)} \omega \in E \text{ for all } |\alpha| \leq n \right\},
\]
equipped with the norm
\[
\|f\|_{E^n_\omega} := \max_{|\alpha| \leq n} \| f^{(\alpha)} \omega \|_E.
\]

Given a weight function system $\W = (w_n)_{n \in \N}$, we define the projective limit spaces
\[
E_\W := \bigcap_{n \in \N} E_{w_n}, \qquad E^\infty_\W := \bigcap_{n \in \N} E^n_{w_n},
\]
equipped with their respective Fréchet space topologies.

In the special case where $E = L^p$ for some $p \in [1, \infty] \cup \{0\}$, we adopt the abbreviations
\[
L^p_\W := (L^p)_\W, \qquad \K^p_\W := (L^p)^\infty_\W.
\]
Moreover, following the notation of Schwartz \cite{S-ThDist}, we write
\[
\B_\W := \K^\infty_\W, \qquad \dot{\B}_\W := \K^0_\W.
\]

This general framework allows us to treat many smooth function spaces appearing in the literature in a unified manner. Below, we provide several illustrative examples.	

\begin{example}

$(i)$ When $\W = \mathbf{1}$, we typically write $\D_E = E^\infty_\mathbf{1}$ (similarly, $\B = (L^\infty)^\infty_{\mathbf{1}}$ and $\dot{\B} = (L^0)^\infty_{\mathbf{1}}$). This is the space of all smooth functions $f$ such that $f$ and all of its derivatives belong to $E$.
In particular, we recover the spaces $\D_{L^p}$, $p \in [1, \infty)$, $\mathcal{B}$, and $\dot{\B}$ introduced by Schwartz \cite{S-ThDist}.

$(ii)$ For $\W = \W_{\text{pol}}$, we find the Schwartz space $\S = E^\infty_{\W_{\text{pol}}}$ \cite{S-ThDist} (see Theorem~\ref{t:IndependenceTIBF}).

$(iii)$ In the case $\W = \W_{\text{exp}}$, we obtain the Hasumi--Silva space $\K_1 = E^\infty_{\W_{\text{exp}}}$ \cite{H-NoteNDimTempUltraDist, H-SP-ThGenFunc} (see Theorem~\ref{t:IndependenceTIBF}). This space is also studied by Schott and Rychkov \cite{Rychkov01, S-FuncSpExpWeighI} in the context of Besov and Triebel–Lizorkin spaces with $A_{\mathrm{loc}}^\infty$-weights.
\end{example}

Let $\varphi \in L^1(\mathbb{R}^d)$. We
write $\mathcal{F}\{\varphi\}$ or $\widehat{\varphi}$
for the Fourier transform of $\varphi$, defined by
\begin{equation*}
    \mathcal{F}\{\varphi\}(\xi) = \widehat{\varphi}(\xi) = (2\pi)^{-d/2} \int_{\mathbb{R}^d} \varphi(x)\, e^{-i x \cdot \xi} \, dx,
\end{equation*}
where $x \cdot \xi$ denotes the standard Euclidean inner product on $\mathbb{R}^d$.
If $\varphi \in \mathcal{D}_L(K)$, then its Fourier transform $\widehat{\varphi}$ satisfies
\begin{equation*}
    \widehat{\varphi}^{(\alpha)}(0) = 0 \quad \text{for all } |\alpha| < L.
\end{equation*}

Since $C_{\rm c}(\R^d) \subseteq E$, each space $E^\infty_\W$ contains $\D$.
Let us now investigate its density.
Of course, in general, one cannot have density: $\D$ is not dense in $\B$.
Things change when the space is defined via the essential part.

\begin{lemma}\label{l:StronglyReduced}
Let $\omega$ be a weight function.
Then
	\[ (E_{\approxim})^{n + d}_{\omega} \subseteq \overline{\D}^{E^n_{\omega}} , \qquad n \in \N . \]
\end{lemma}

\begin{proof}
Let $\chi, \psi \in \D$ be such that $0 \leq \chi, \psi \leq 1$, $\int \chi = 1$, and $\psi \equiv 1$ on $[-1,1]^d$.
For any $R \geq 1$, define $\chi_R := R^d \chi(R \, \cdot)$ and $\psi_R := \psi(\cdot / R)$.

Fix $f \in (E_{\approxim})^{n + d}_{\omega}$ and $\varepsilon > 0$. 
Since $E$ is solid,
\begin{align*}
\| \partial^\alpha [f - \psi_R f] \cdot \omega \|_{E}
&\leq \| (1 - \psi_R) f^{(\alpha)} \cdot \omega \|_{E} + \sum_{0 < \beta \leq \alpha} {\alpha \choose \beta} \| \psi_R^{(\beta)} f^{(\alpha - \beta)} \cdot \omega \|_{E} \\
&\leq 2^{|\alpha|} \| \psi \|_{(L^\infty)^n_1} \cdot \max_{|\gamma| \leq n} \| (1 - 1_{[-R, R]^d}) f^{(\gamma)} \cdot \omega \|_{E}
\end{align*}
for each multi-index $\alpha$ with $|\alpha| \leq n$.
Since $f^{(\gamma)} \cdot \omega \in E_{\approxim}$ for all $|\gamma| \leq n$, we have
\[
\lim_{R \to \infty}
\left(
\max_{|\gamma| \leq n} \| (1 - 1_{[-R, R]^d}) f^{(\gamma)} \cdot \omega \|_{E} 
\right)=0.
\]
Thus, there exists $R^1_\varepsilon \geq 1$ such that
\begin{equation}\label{eq:250407-51}
\| f - \psi_{R^1_\varepsilon} f \|_{E^n_\omega}
=\sup_{|\alpha| \le n}
\| \partial^\alpha [f - \psi_R f] \cdot \omega \|_{E} \leq \frac{\varepsilon}{2}.
\end{equation}

Remark that $\psi_{R^1_\varepsilon} f \in (L^1)^{n + d}_\omega$.
In particular, $\psi_{R^1_\varepsilon} f \in C^n(\R^d)$.
Now, using the local boundedness of $\omega$,
 the continuous inclusion $C_{\rm c}(\mathbb{R}^d) \subseteq E$, and
the fact that $E$ is solid, we note that for any $g \in C_{\rm c}(\mathbb{R}^d)$,
\[
\lim_{R \to \infty}\left(
\sup_{t \in \supp \chi} \| (g - T_{t/R} g) \cdot \omega \|_{E} 
\right)=0.
\]
Consequently,
since $\int\chi=1$,
\[
\lim_{R \to \infty}
\| \psi_{R^1_\varepsilon} f - \chi_R * (\psi_{R^1_\varepsilon} f) \|_{E^n_\omega} \leq 
C\lim_{R \to \infty}
\left(
\sup_{t \in \supp \chi} \| \psi_{R^1_\varepsilon} f - T_{t/R}(\psi_{R^1_\varepsilon} f) \|_{E^n_\omega}\right)=0.
\]
Hence, there exists $R^2_\varepsilon \geq 1$ such that
\begin{equation}\label{eq:250407-52}
\| \psi_{R^1_\varepsilon} f - \chi_{R^2_\varepsilon} * (\psi_{R^1_\varepsilon} f) \|_{E^n_\omega} \leq \frac{\varepsilon}{2}.
\end{equation}
Combining (\ref{eq:250407-51}) and (\ref{eq:250407-52}), we obtain
\[
\| f - \chi_{R^2_\varepsilon} * (\psi_{R^1_\varepsilon} f) \|_{E^n_\omega} \leq \varepsilon.
\]
Since $\chi_{R^2_\varepsilon} * (\psi_{R^1_\varepsilon} f) \in \D$ and $\varepsilon > 0$ was arbitrary, the proof is complete.
\end{proof}

\begin{corollary}
	\label{c:DenseInclusionD}
	It holds
	\[ (E_{\approxim})^\infty_\W = \overline{\D\,}^{E^\infty_\W} . \]
\end{corollary}

\begin{proof}
	Directly by Lemma \ref{l:StronglyReduced}. 
\end{proof}

\subsection{The role of the solid TIBF of bounded type}
We investigate the role of the solid TIBF of bounded type \( (E, \|\cdot\|_E) \) 
in defining the space \( E^\infty_{\mathcal{W}} \) associated with a weight function system \( \mathcal{W} \).
When \( \mathcal{W} = \mathbf{1} \), the resulting spaces \( \mathcal{D}_E \) 
may differ drastically depending on the underlying space \( E \); 
see \cite[Theorem~6.5]{D-N-SeqSpRepTMIB}.
However, if consecutive weight functions in \( \mathcal{W} \) dominate each other in a suitable sense 
(in particular, under conditions~\((\condwN)\) and~\((\condN)\)), 
some or all of these spaces may coincide.

We first provide a coarse localization of the space \( E^n_{w_n} \).
\begin{lemma}
    \label{l:InclusionL^1andL^infty}
    Let $\mathcal{W}$ be a weight function system satisfying $(\mathrm{\condwM})$. 
    For any $n \in \mathbb{N}$, there exists an integer $m \geq n$ such that the following continuous inclusions hold:
    \begin{align}
    	(W(L^1, E_{\rm seq}))^m_{w_m} &\subseteq E^n_{w_n} , \label{eq:W(L1, Ed)_to_E} \\
        (L^1)^{m}_{w_m} &\subseteq E^{n}_{w_n}, \label{eq:L1_to_E} \\
        E^m_{w_m} &\subseteq (L^\infty)^{n}_{w_n}, \label{eq:E_to_Linfty} \\
        (E_{\approxim})^m_{w_m} &\subseteq (L^0)^n_{w_n}. \label{eq:Eappr_to_L0}
    \end{align}
\end{lemma}

\begin{proof}
We start with two remarks.
First, the continuity of these inclusions follows from the closed graph theorem since all involved spaces are continuously contained in $L^1_{\mathrm{loc}}(\mathbb{R}^d)$.
Second, it suffices to establish \eqref{eq:W(L1, Ed)_to_E} solely. 
Indeed, 
$F \subseteq W(L^1, F_{\rm seq})$ for any solid TIBF of bounded type $(F, \|\cdot\|)$ by Lemma \ref{l:W(L^1, E_d)ConvSp}.
Also, $\ell^1 \subseteq E_{\rm seq} \subseteq \ell^\infty$ according to \cite[Lemma 3.5(b)]{F-G-BanachSpIntGroupRepAtomicDecomp}. 
This readily implies $((E_{\approxim})_{\rm seq} \subseteq c_0$.
These observation then show that \eqref{eq:L1_to_E}--\eqref{eq:Eappr_to_L0} would hold.

We prove \eqref{eq:W(L1, Ed)_to_E} using the Schwartz parametrix method. Let us begin with the setup. Choose $\chi \in \mathcal{D}$ such that $1_{\overline{B}(0, 1/2)} \leq \chi \leq 1_{\overline{B}(0, 1)}$. 
As in \cite{S-ThDist}, let $F_\ell \in L^1_{\mathrm{loc}}(\mathbb{R}^d)$ be the fundamental solution of $\Delta^\ell$ for a positive integer $\ell$. 
Then
\begin{equation} \label{eq:parametrix_phi}
    \varphi_\ell=\Delta^\ell(\chi F_\ell) - \delta =\Delta^\ell[(1-\chi)F_\ell] \in \mathcal{D}.
\end{equation}
Consequently, for any $f \in (W(L^1, E_{\rm seq}))^m_{w_m}$,
\begin{equation} \label{eq:parametrix_formula}
    f = (\Delta^\ell f) * (\chi F_\ell) - f * \varphi_\ell.
\end{equation}
Choosing $\ell$ sufficiently large, say $\ell > d$, ensures that $\chi F_\ell \in C_{\mathrm{c}}(\mathbb{R}^d)$.

Next, for each $n \in \mathbb{N}$, choose
 $k_n \geq n$ and a constant $C_n > 0$ such that
\begin{equation} \label{eq:weight_dominance}
    w_n(x + y) \leq C_n w_{k_n}(x), \quad \text{for all } x \in \mathbb{R}^d,\ y \in \overline{B}(0, 1).
\end{equation}
As a consequence, we obtain the inequality
\begin{equation} \label{eq:convolution_derivative_bound}
    |\partial^\alpha [g * \phi]| w_n \leq C [|g^{(\alpha)}| w_{k_n}] * |\phi|
\end{equation}
for all $g \in (W(L^1, E_{\rm seq}))^{k_n}_{w_{k_n}}$, $\phi \in C_{\mathrm{c}}(\overline{B}(0, 1))$, and $|\alpha| \leq n$.

Putting $m = k_n + 2\ell$ and by using the solidity of $E$, we obtain \eqref{eq:W(L1, Ed)_to_E} from combining \eqref{eq:ConvAmalgam}, \eqref{eq:parametrix_formula}, and \eqref{eq:convolution_derivative_bound}, 
%
%
%
\end{proof}

Proposition \ref{p:LowerUpperInclusions} serves as a localization
of $E^\infty_\W$, and follows directly from Lemma \ref{l:InclusionL^1andL^infty}:
The former half is due to 
(\ref{eq:W(L1, Ed)_to_E}).
The latter half is (\ref{eq:L1_to_E})--(\ref{eq:Eappr_to_L0}).
In essence, it states that the sequence space $E_{\rm seq}$ already fully determines the space $E^\infty_\W$.

\begin{proposition}
    \label{p:LowerUpperInclusions}
    Let $\W$ be a weight function system satisfying $(\condwM)$. 
    Then
    	\[ E^\infty_\W = (W(L^1, E_{\rm seq}))^{\infty}_\W . \]
    Moreover, we have the continuous inclusions
    	\[ \K^1_\W \subseteq E^\infty_\W \subseteq \B_\W \] 
   and 
   	\[ \K^1_\W \subseteq (E_{\approxim})^\infty_\W \subseteq \dot{\B}_\W . \]
\end{proposition}
%

Proposition \ref{p:IndependenceSemiLoc} below
states that, under certain conditions on a weight function system $\mathcal{W}$, 
regardless of the property of the underlying spaces $E$,
the Fr\'echet space $E^\infty_\mathcal{W}$ remains unchanged whether it is defined using the essential part $(E_{\approxim})^\infty_\mathcal{W}$ or the multiplication module completion $(E_{\semloc})^\infty_\mathcal{W}$. 

\begin{proposition}
	\label{p:IndependenceSemiLoc}
	Let $\W$ be a weight function system satisfying $(\condwM)$ and $(\condwN)$.
	Then $E^\infty_\W = (E_{\approxim})^{\infty}_\W = (E_{\semloc})^{\infty}_\W$ as Fr\'{e}chet spaces.
\end{proposition}

\begin{proof}
	By Lemma \ref{l:DotandSemilocProp}$(i)$, we have the continuous inclusions $(E_{\approxim})^{\infty}_\W \subseteq E^\infty_\W \subseteq (E_{\semloc})^{\infty}_\W$.
	We may assume the elements of $\W$ are continuous by Lemma \ref{l:SmoothWeightFuncSyst}.
Using $(\condwN)$,
we choose for each $n \in \N$ some $m_n \geq n$ such that $w_n / w_{m_n} \in C_0(\mathbb{R}^d)$. 
	Take any $f \in (E_{\semloc})^\infty_\W$.
	Since $(E_{\approxim})_{\semloc} = E_{\semloc}$ by Lemma \ref{l:DotandSemilocProp}$(iii)$, it follows from Proposition \ref{p:E=MEiffSemiLocal} that for any $n \in \N$ and $\alpha \in \N^d$
		\[ f^{(\alpha)} w_n = [f^{(\alpha)} w_{m_n}] \cdot \frac{w_n}{w_{m_n}} \in E_{\approxim} . \]
	This shows that $(E_{\semloc})^\infty_\W \subseteq (E_{\approxim})^\infty_\W$, so that the three spaces coincide as sets.
Their equality as Fréchet spaces follows from the open mapping theorem.
\end{proof}

The next lemma connects spaces of Gelfand-Shilov-type
generated by two different TIBFs.
\begin{lemma}
    \label{l:InclusionDifferentTIBF}
    Let $\W$ be a weight function system satisfying $(\condwM)$ and $(\condN)$.
    Let $(E, \|\cdot\|_E)$ and $(F, \|\cdot\|_F)$ be two solid TIBFs of bounded type. Then for any $n \in \N$, there exists an integer $m \geq n$ such that $F^m_{w_m} \subseteq E^{n}_{w_n}$ continuously.
\end{lemma}

\begin{proof}
    By \eqref{eq:L1_to_E} and \eqref{eq:E_to_Linfty},
we may assume that $E=L^1$ and that $F=L^\infty$, that is, it suffices to show that $(L^\infty)^m_{w_m} \subseteq (L^1)^{n}_{w_n}$ for some $m \geq n$. For each $n \in \N$,
using $(\condN)$, we choose $m=m_n \geq n$ such that $w_n / w_{m_n} \in L^1$. Then, for $f \in (L^\infty)^{m_n}_{w_{m_n}}$, we have
    \[
        \int_{\mathbb{R}^d} |f^{(\alpha)}(x)| w_n(x) dx 
\leq \|w_n / w_{m_n}\|_{L^1} \|f\|_{(L^\infty)^{n}_{w_{m_n}}}
\leq \|w_n / w_{m_n}\|_{L^1} \|f\|_{(L^\infty)^{m_n}_{w_{m_n}}}
    \]
for all $\alpha \in {\mathbb N}^d$ with $|\alpha| \le n$.
This proves
that $F^m_{w_m} \subseteq E^{n}_{w_n}$ continuously.
\end{proof}

For any weight function system $\W$ satisfying $(\condwM)$, the property $(\condN)$ is equivalent to the definition of $E^\infty_\W$ being independent of the solid TIBF of bounded type $E$.

\begin{theorem}
    \label{t:IndependenceTIBF}
    Let $\W$ be a weight function system satisfying $(\condwM)$. The following statements are equivalent{\rm:}
    \begin{itemize}
        \item[$(i)$] $\W$ satisfies $(\condN)$.
        \item[$(ii)$] $E^\infty_\W = F^\infty_\W$ as sets for any two solid TIBFs of bounded type $E, F$.
        \item[$(iii)$] $E^\infty_\W = F^\infty_\W$ as Fr\'echet spaces for any two solid TIBFs of bounded type $E, F$.
        \item[$(iv)$] $\mathcal{K}^1_\W = \B_\W$ as sets.
        \item[$(v)$] $\mathcal{K}^1_\W = \B_\W$ as Fr\'echet spaces.
    \end{itemize}
\end{theorem}

\begin{proof}
    The implications 
    \[
(iii) \Rightarrow (ii) \Rightarrow (iv) \quad \text{and} \quad (iii) \Rightarrow (v) \Rightarrow (iv)
    \]
    are trivial.
Meanwhile
the implications 
$(i) \Rightarrow (v) \Rightarrow (iii)$
follow from Proposition~\ref{p:LowerUpperInclusions} and Lemma~\ref{l:InclusionDifferentTIBF}.
It thus suffices to show that $(iv) \Rightarrow (i)$.

Assume that $\mathcal{K}^1_\W = \B_\W$
as sets.
We define the Fr\'echet spaces
	\[ \lambda^p_\W = \left\{ (c_{k})_{k \in \Z^d} \in \C^{\Z^d} \mid (c_k w_n(k))_{k \in \Z^d} \in \ell^p \text{ for all } n \in \N \right\} , \qquad p \in \{1, \infty\} . \]
    Since $\ell^1$ is embedded into $\ell^\infty$,
$\lambda^1_\W \subseteq \lambda^\infty_\W$.
Take any $\chi \in \D([-1/2, 1/2]^d)$ such that $\int \chi^2 = 1$.
Using $(\condwM)$, one easily verifies the linear mappings
	\[ S : \lambda^\infty_\W \to \K^\infty_\W , \qquad (c_k)_{k \in \Z} \mapsto \sum_{k \in \Z^d} c_k T_k \chi \]
and
	\[ T : \K^1_\W \to \lambda^1_\W , \qquad 
	f \mapsto \left(\int_{\R^d} f T_k \chi \right)_{k \in \Z^d} \]
are well
defined. As $\K^1_\W = \K^\infty_\W$, we then see that $T \circ S$ gives the identity $\lambda^\infty_\W = \lambda^1_\W$.
Then $\W$ must satisfy $(\condN)$ by \cite[Proposition 2.1]{D-N-V-NuclGSSpKernThm}.
\end{proof}

\begin{remark}
\
\begin{enumerate}
\item[$(i)$]
The statements $(i)$--$(v)$ in Theorem \ref{t:IndependenceTIBF} are also equivalent to: \emph{$E^\infty_\W$ is nuclear for some/any solid TIBF of bounded type $E$}.
Indeed, if $\W$ satisfies $(\condN)$, then this was already shown for $\K^\infty_\W = E^\infty_\W$ in \cite[p.~181]{G-S-GenFunc3} under some extra conditions, but can also be proven more generally as in \cite[Theorem 5.1]{D-N-V-NuclGSSpKernThm}.
On the other hand, if some $E^\infty_\W$ is nuclear, one verifies similarly as in \cite[Corollary 4.8]{D-N-V-NuclGSSpKernThm} that the embedding $\lambda^1_\W \subseteq \lambda^\infty_\W$ factors through $E^\infty_\W$.
By a result due to Petzsche \cite{P-Nukl}, see \cite[Lemma 5.5]{D-N-V-NuclGSSpKernThm}, it would follow that $\lambda^1_\W$ is nuclear.
Then $\W$ would satisfy $(\condN)$ by \cite[Proposition 2.1]{D-N-V-NuclGSSpKernThm}.
\item[$(ii)$]
Let $\W={\bf 1}$ and $1<q < p<\infty$.
Let $\Lambda$ be a discrete subset of $\R^d$ such that $Z=\bigcup\limits_{\lambda \in \Lambda}(\lambda+[0,1]^n)$ is the disjoint union of infinitely many closed cubes of volume $1$.
If the elements of $\Lambda$ are far enough apart, then
$\chi_Z \in {\mathcal M}^p_q$.
Choose a function $\psi \in \D([0,1]^n) \setminus \{0\}$.
Then the example of
$\sum\limits_{\lambda \in \Lambda} T_\lambda \psi \in 
\D_{\mathcal{M}^p_q}\setminus \D_{L^p}$
shows it can happen that
$\D_{\mathcal{M}^p_q} \ne \D_{L^p}$.
\end{enumerate}
\end{remark}
%

\subsection{Convolution}

We investigate the smoothing effect of the convolution with
$C_{\rm c}(K)$ and $\D(K)$.
We prove a general estimate.
\begin{lemma}
	\label{l:ConvGenOmega}
Let $K \Subset \mathbb{R}^d$ and $C > 0$.
	Let $\omega, \nu : \mathbb{R}^d \to (0, \infty)$ locally integrable functions such that $\omega(x + y) \leq C \nu(x)$ for all $x \in \mathbb{R}^d$ and $y \in K$.
	Then for any $f \in E_\nu$ and $\phi \in C_{\rm c}(K)$ we have $f * \phi \in E_\omega$ and 
		\begin{equation}
			\label{eq:ConvGenOmega} 
			\|[f * \phi] \omega\|_E \leq C  \|f \nu\|_E \|\phi\|_{L^1} .  
		\end{equation}
\end{lemma}

\begin{proof}
We have $|f * \phi| \omega \leq C |f \nu| * |\phi|$
	for any $f \in E_\nu$ and $\phi \in C_{\rm c}(K)$.
	Then $[f * \phi] \omega \in E$ by the solidity of $E$, and \eqref{eq:ConvGenOmega} follows from \eqref{eq:ModuleBounds}.
\end{proof}

We investigate the continuity of convolution operators in $E_\W$.
\begin{proposition}
    \label{p:EWConvCompact}
    Let $\W$ be a weight function system satisfying $(\condwM)$.
    For any compact subset $K \Subset {\mathbb R}^d$, the convolution mappings
    \[ * : E_\W \times C_{\rm c}(K) \to E_\W \quad \text{and} \quad * : E_\W \times \D(K) \to E^\infty_\W \]
    are well
defined and continuous.
\end{proposition}

\begin{proof}
	Since
$\W$ satisfies $(\condwM)$, the continuity of 
the first convolution mapping
$* : E_\W \times C_{\rm c}(K) \to E_\W$ follows directly from Lemma \ref{l:ConvGenOmega}.
	The continuity of
the second convolution mapping
 $* : E_\W \times \D(K) \to E^\infty_\W$ is now an immediate consequence
 of the former.
\end{proof}

Here and below, we mean by the well-definedness of
the convolution $* : E_\W \times L^1_\V \to E_\W$ that
the restriction of the continuous map $* : E \times L^1 \to E$
to $E_\W \times L^1_\V$ has its image in $E_\W$.
Similar conventions will be used for other restrictions.

Next, we show that the notion of $\V$-moderateness is natural in this context.
\begin{theorem}
    \label{t:EWConvModerate}
    Let $\V, \W$ be weight function systems satisfying $(\condwM)$. The following statements are equivalent{\rm:}
    \begin{itemize}
        \item[$(i)$] $\W$ is $\V$-moderate{\rm;}
        \item[$(ii)$] the convolution $* : E_\W \times L^1_\V \to E_\W$ is well-defined{\rm;}
        \item[$(iii)$] the convolution $* : E_\W \times \K^1_\V \to E^\infty_\W$ 
is well-defined.
    \end{itemize}
    If these statements hold, the mappings in $(ii)$ and $(iii)$ are continuous.
\end{theorem}

\begin{proof}
Before we show the equivalences, we note that if either of the mappings in $(ii)$ or $(iii)$ is well-defined, it is automatically continuous 
by the closed graph theorem since all involved spaces are Fr\'echet spaces.

    $(i) \Rightarrow (ii)$: Fix $n \in \mathbb{N}$. 
By assumption we can choose
$m \geq n$ and $C > 0$ so that $w_n(x + y) \leq C w_m(x) v_m(y)$ for all $x, y \in {\mathbb R}^d$. 
Then
   	\[ |f * \psi| w_n \leq C [|f w_m| * |\psi v_m|] \]
   for any $f \in E_\W$ and $\psi \in L^1_\V$.
    By the solidity of $E$,
    \[
        \| |f * \psi| w_n \|_E \leq C \|f w_m\|_E \|\psi v_m\|_{L^1} . 
    \]
    We obtain the well-definedness of $E_\W \times L^1_\V \to E_\W$.

    $(ii) \Rightarrow (iii)$: Since $\partial^\alpha [f * \psi] = f * \psi^{(\alpha)}$ for all $f \in E_\W$, $\psi \in \K^1_\V$, and $\alpha \in \mathbb{N}^d$, this follows immediately.

    $(iii) \Rightarrow (i)$: Fix
$n \in \N$. 
Condition $(\condwM)$ allows us to
choose $n' \geq n$ and $C > 0$ so that $T_y w_{n} \leq C w_{n'}$ for any $y \in [-1, 2]^d$.
		Assuming 
    $(iii)$,
we see that the map $* : E_\W \times \K^1_\V \to E^\infty_\W$ is continuous. Consequently, there exists some $m' \geq n'$ and $C' > 0$ such that
\begin{equation}\label{eq:250407-5} 
\| (f * \psi) w_{n'} \|_E \leq C' \| f w_{m'} \|_E \cdot \max_{|\alpha| \leq m'}  \| \partial^\alpha\psi v_{m'} \|_{L^1} , \qquad f \in E_\W, \psi \in \K^1_\V.
\end{equation}
Once again, condition $(\condwM)$ allows us to choose $m \geq m'$ and $C'' > 0$ such that 
\begin{equation}\label{eq:250407-2}
T_{-y} w_{m'} \leq C'' w_m \quad \text{and} \quad T_{-y} v_{m'} \leq C'' v_m , \qquad \forall y \in [-1, 2]^d . 
\end{equation} 
		Pick a non-negative function
$\phi \in \D([-1,0]^d)$ with $\int \phi = 1$.
		Note that 
\begin{equation}\label{eq:250407-41}
1_{[0, 2]^d} * \phi \geq 1_{[0, 1]^d}.
\end{equation}
Fix $x,y \in {\mathbb R}^d$.
		We then define the functions
			\[ f_x := \frac{1}{w_m(x)} T_x 1_{[0, 2]^d}, 
\qquad \psi_y := \frac{1}{v_m(y)} T_y \phi , \qquad x, y \in \mathbb{R}^d . \]
Then $f_x \in E_\W$ and $\psi_y \in \K^1_\V$.  
Since\begin{equation}\label{eq:250407-41a}
	f_x w_{m'} \leq C'' w_m(x) f_x = C'' T_x 1_{[0, 2]^d}
\end{equation}
due to (\ref{eq:250407-2}), it follows from the solidity and translation invariance of $E$ that
\begin{equation}\label{eq:250505-1}
	\| f_x w_{m'} \|_E \leq C'' C_0 \| 1_{[0, 2]^d} \|_E , \qquad x \in \mathbb{R}^d .
\end{equation}
Similarly, 
\begin{equation}\label{eq:250505-2}
	\|\psi_y^{(\alpha)} v_{m'}\|_{L^1} \leq C'' \|\phi^{(\alpha)}\|_{L^1}
\end{equation}
holds for any $\alpha \in \N^d$ and $y \in \mathbb{R}^d$, again due to (\ref{eq:250407-2}).

Now note that
\[
	f_x * \psi_y = \frac{1}{w_m(x) v_m(y)} T_{x + y} [1_{[0, 2]^d} * \phi] , \qquad x, y \in \mathbb{R}^d .
\]
The support of $1_{[0,2]^d}*\phi$ is contained in $[-1,2]^d$,
so we deduce from (\ref{eq:250407-41})
\begin{align*}
	\frac{w_n(x + y)}{w_m(x) v_m(y)} \| 1_{[0, 1]^d} \|_E 
	&\leq \frac{w_n(x + y)}{w_m(x) v_m(y)} \| 1_{[0, 2]^d} * \phi \|_E \\
	&\leq C_0 \left\| \frac{1}{w_m(x) v_m(y)} T_{x + y} [1_{[0, 2]^d} * \phi] \cdot w_n(x + y) \right\|_E \\
	&\leq C C_0 \left\| \frac{1}{w_m(x) v_m(y)} T_{x + y} [1_{[0, 2]^d} * \phi] \cdot w_{n'} \right\|_E \\
	&= C C_0 \| (f_x * \psi_y) w_{n'} \|_E .
\end{align*}
We now estimate the norm above on the most right-hand side using (\ref{eq:250407-5}),
(\ref{eq:250505-1}) and (\ref{eq:250505-2}):
\begin{align*}
	\| (f_x * \psi_y) w_{n'} \|_E 
	&\leq C' \| f_x w_{m'} \|_E \cdot \max_{|\alpha| \leq m'} \| \psi_y^{(\alpha)} v_{m'} \|_{L^1} \\
	&\leq C' (C'')^2 C_0 \| 1_{[0, 2]^d} \|_E \cdot \max_{|\alpha| \leq m'} \| \phi^{(\alpha)} \|_{L^1} .
\end{align*}
Putting everything together, we obtain a constant $\widetilde{C} > 0$ such that
%
\[
	w_n(x + y) \leq \widetilde{C} w_m(x) v_m(y), \qquad x, y \in \mathbb{R}^d,
\]
implying that $\W$ is $\V$-moderate.
\end{proof}

\section{The convolutor and multiplier spaces}
\label{sec:ConvMultSp}
For any weight function system $\mathcal{W}$, we can define the continuous linear mappings:
\begin{equation}\label{eq:convolution_multiplication}
    * : E_\mathcal{W} \times \mathcal{D} \to E^\infty_\mathcal{W}, \qquad 
    \cdot : E^\infty_\mathcal{W} \times \mathcal{B} \to E^\infty_\mathcal{W}.
\end{equation}
The first operation represents convolution, as shown in Proposition \ref{p:EWConvCompact}, and the second denotes multiplication, which can be easily verified.

This section aims to identify the largest possible spaces of distributions for which convolution with \(\mathcal{D}\) (as discussed in Section~\ref{s5.1}) and multiplication with \(\mathcal{B}\) (as discussed in Section~\ref{s5.2}) map within \(E^\infty_\mathcal{W}\). We also investigate how convolution and multiplication behave when the test function spaces are replaced by larger or smaller ones, respectively.

\subsection{The convolutor spaces}
\label{s5.1}

Let 
$\mathcal{W}$ be a weight function system. We consider the convolutor space
\[
\mathcal{O}'_C(\mathcal{D}, E_\mathcal{W}) := \{ f \in \mathcal{D}' \,:\, f * \varphi \in E_\mathcal{W} \text{ for all } \varphi \in \mathcal{D} \} .
\]
Since $\mathcal{D}$ is closed under differentiation, so is $\mathcal{O}'_C(\mathcal{D}, E_\mathcal{W})$. 
Therefore, one has that $f * \varphi \in E^\infty_\W$ for any $f \in \mathcal{O}'(\D, E_\W)$ and $\varphi \in \D$, that is,
\begin{equation}\label{eq:250224-3}
\mathcal{O}'_C(\mathcal{D}, E_\mathcal{W}) = \{ f \in \mathcal{D}' \,:\, f * \varphi \in E^\infty_\mathcal{W} \text{ for all } \varphi \in \mathcal{D} \} .
\end{equation}

We directly obtain the following inclusion 
from Proposition \ref{p:EWConvCompact}.

\begin{lemma}\label{l:InclusionEWConvSp}
Let $\mathcal{W}$ be a weight function system satisfying $(\condwM)$. 
Then $E_\W \subseteq \OC^\prime(\D, E_\W)$.
\end{lemma}

Let $\V$ now be another weight function system satisfying $(\condwM)$. 
Then $\check{\V} = (\check{v}_n)_{n \in \N}$ satisfies $(\condwM)$ as well so that $\K^1_{\check{\V}}$ is translation-invariant.
For any $f \in (\K^1_{\check{\V}})^\prime$ and $\varphi \in \K^1_\V$, we may now define the convolution
	\[ f * \varphi(x) := \ev{f}{T_x \check{\varphi}} . \]
This gives a linear mapping 
$(\K^1_{\check{\V}})^\prime \times \K^1_\V \to C(\mathbb{R}^d)$.	

We consider another convolutor space
	\[ \OC^\prime(\K^1_\V, E_\W) := \{ f \in (\K^1_{\check{\V}})^\prime \,:\, f * \varphi \in E_\W \text{ for all } \varphi \in \K^1_\V \} . \]
As in (\ref{eq:250224-3}),
we have
\begin{equation}\label{eq:250224-31}
 \OC^\prime(\K^1_\V, E_\W) = \{ f \in (\K^1_{\check{\V}})^\prime \,:\, f * \varphi \in E_\W^\infty \text{ for all } \varphi \in \K^1_\V \} . \end{equation}
As $\D$ is dense in $\K^1_\V$, by Corollary \ref{c:DenseInclusionD}
with $E=L^1$, we then have the embedding
	\begin{equation}\label{eq:250407-71} f \in \OC^\prime(\K^1_\V, E_\W) \mapsto f_{|\mathcal{D}} \in \OC^\prime(\D, E_\W) . \end{equation}
\begin{lemma}
Let $\mathcal{W}$ and $\mathcal{V}$ be weight function systems satisfying \((\condwM)\). 
Assume that $\mathcal{W}$ is $\mathcal{V}$-moderate. Then
\[
E_\mathcal{W} \subseteq \OC'(\mathcal{K}^1_\mathcal{V}, E_\mathcal{W}).
\]
\end{lemma}

\begin{proof}
Simply use Theorem \ref{t:EWConvModerate}(iii).
\end{proof}

Any $f \in \OC^\prime(\D, E_\W)$, respectively $f \in \OC^\prime(\K^1_\V, E_\W)$, can be seen as an element of $L(\D, E^\infty_\W)$, respectively $L(\K^1_\V, E^\infty_\W)$, via the embedding
	\[ f \mapsto *_f , \quad \text{where } *_f(\varphi) := f * \varphi . \]
In fact, we can fully characterize the convolutor spaces as those operators that commute with translation, as observed by H\"{o}rmander
\cite{Hormander1960}.

\begin{proposition} For \( L \in L(\D, E^\infty_\W) \), respectively \( L \in L(\K^1_\V, E^\infty_\W) \), there exists an element  
\( f \in \OC'(\D, E_\W) \), respectively \( f \in \OC'(\K^1_\V, E_\W) \), such that \( L = *_f \)  
if and only if  
\begin{equation}\label{eq:CommTrans}
T_x \circ L = L \circ T_x \quad \text{for every } x \in \R^d .
\end{equation}
\end{proposition}

\begin{proof}
From the definition of the convolution,
any convolution operator $*_f$ satisfies \eqref{eq:CommTrans}
in both cases.
On the other hand, let $L \in L(\D, E^\infty_\W)$ satisfy \eqref{eq:CommTrans} and define $f \in \D^\prime$ by $\ev{f}{\phi} = \ev{\delta}{L(\check{\phi})}$ for $\phi \in \D$.
Then, for any $\varphi \in \D$ and $x \in \R^d$,
	\[ f * \varphi(x) = \ev{f}{T_x \check{\varphi}} = \ev{\delta}{L(T_{-x} \varphi)} = \ev{\delta}{T_{-x} L(\varphi)} = L(\varphi)(x) . \] 
Hence, $f * \varphi = L(\varphi) \in E^\infty_\W$. Thus, $f \in \OC^\prime(\D, E_\W)$.
If $L \in L(\K^1_\V, E^\infty_\W)$, the same argument yields $f \in \OC^\prime(\K^1_\V, E_\W)$.
\end{proof}

From the definition,
we deduce that
the operator
$f \mapsto f|_{\mathcal D}$ is injective from 
$
  \OC'(\mathcal{K}^1_\mathcal{V}, E_\mathcal{W})$
  to
  $
  \OC'(\mathcal{D}, E_\mathcal{W})$.
We now investigate under which conditions the spaces $\OC'(\mathcal{K}^1_\mathcal{V}, E_\mathcal{W})$ and $\OC'(\mathcal{D}, E_\mathcal{W})$ coincide. 
If  $\mathcal{W}$ is $\mathcal{V}$-moderate, this is indeed the case. The following result generalizes \cite[Proposition~6.2]{D-V-TopPropConvSpSTFT}.

\begin{theorem}\label{t:ExtensionOC}
Let $\mathcal{V}$ and $\mathcal{W}$ be weight function systems satisfying \((\condwM)\). Then
\begin{equation}\label{eq:ExtensionOC}
  \OC'(\mathcal{K}^1_\mathcal{V}, E_\mathcal{W}) = \OC'(\mathcal{D}, E_\mathcal{W})
\end{equation}
if and only if $\mathcal{W}$ is $\mathcal{V}$-moderate.
\end{theorem}

Before proving the theorem, we fix some notation.  
For $r \in \mathbb{N}$ and $K \Subset \mathbb{R}^d$, we denote by $\mathcal{D}^r(K)$ the space of all $C^r$-functions on $\mathbb{R}^d$ supported in $K$.

\begin{proof}
Suppose first that \eqref{eq:ExtensionOC} holds. By Lemma~\ref{l:InclusionEWConvSp}, we have
\[
E_\mathcal{W} \subseteq \OC'(\mathcal{K}^1_\mathcal{V}, E_\mathcal{W}) .
\]
In particular, as in (\ref{eq:250224-31}), the convolution map
\[
* : E_\mathcal{W} \times \mathcal{K}^1_\mathcal{V} \to E^\infty_\mathcal{W}
\]
is well
defined. Hence, by Theorem~\ref{t:EWConvModerate}, it follows that $\mathcal{W}$ is $\mathcal{V}$-moderate.

Conversely, suppose that $\mathcal{W}$ is
$\mathcal{V}$-moderate. To prove \eqref{eq:ExtensionOC}, it suffices to show
\[
\OC'(\mathcal{D}, E_\mathcal{W}) \subseteq \OC'(\mathcal{K}^1_\mathcal{V}, E_\mathcal{W}),
\]
in view of \eqref{eq:250407-71}. Fix any $f \in \OC'(\mathcal{D}, E_\mathcal{W})$. By the closed graph theorem, for each compact set $K \Subset \mathbb{R}^d$, the mapping
\[
*_{f} : \mathcal{D}(K) \to E_\mathcal{W}, \quad \phi \mapsto f * \phi
\]
is continuous. Let $K=\overline{B}(0,1)$. Then, for each $n \in \mathbb{N}$, there exist $r_n \in \mathbb{N}$ and $C_n > 0$ such that
\begin{equation}\label{eq:*fineq}
\| f * \phi \|_{E_{w_n}} \leq C_n \sup_{\substack{\alpha \in \mathbb{N}^d \\ |\alpha| \le r_n}} \| \partial^\alpha \phi \|_{L^\infty(\overline{B}(0,1))}, \quad \phi \in \mathcal{D}(\overline{B}(0,1)).
\end{equation}
By density, \eqref{eq:*fineq} extends to all $\phi \in \mathcal{D}^{r_n}(\overline{B}(0, 1/2))$, so in particular, $f * \phi \in E_{w_n}$ for such $\phi$.

Since $\mathcal{W}$ is
assumed  $\mathcal{V}$-moderate, for each $n \in \mathbb{N}$ there exist $m_n \geq n$ and a constant $C_n > 0$ such that
\[
w_n(x + y) \leq C_n w_{m_n}(x) v_{m_n}(y), \quad x, y \in \mathbb{R}^d.
\]
As in the proof of Theorem~\ref{t:EWConvModerate}$(i)\Longrightarrow(ii)$, this implies that
\[
\varphi * \psi \in E_{w_n}, \quad \text{for all } \varphi \in E_{w_{m_n}},\ \psi \in L^1_{v_{m_n}}.
\]

In a 
similar fashion as in Lemma~\ref{l:InclusionL^1andL^infty}, we can find
a large enough integer $\ell_n$ and functions 
$\chi_{m_n,0} \in \mathcal{D}^{r_{m_n}}(\overline{B}(0,1/2))$ and
$\chi_{m_n,1} \in \mathcal{D}(\overline{B}(0,1/2))$ such that
\[
f = \Delta^{\ell_n} (f * \chi_{m_n,0}) + f * \chi_{m_n,1} =: \Delta^{\ell_n} \varphi_{m_n,0} + \varphi_{m_n,1},
\]
where $\varphi_{m_n,0}, \varphi_{m_n,1} \in E_{w_{m_n}}$.
Then, for any $\phi \in \mathcal{K}^1_{\mathcal{V}}$, we have
\[
f * \phi = \varphi_{m_n, 0} * \Delta^{\ell_{n}} \phi + \varphi_{m_n, 1}* \phi \in E_{w_n},
\]
for all $n \in \mathbb{N}$, hence $f * \phi \in E_\mathcal{W}$. We conclude that $f \in \OC'(\mathcal{K}^1_\mathcal{V}, E_\mathcal{W})$, as desired.
\end{proof}

\begin{remark}
Following the proof of Theorem \ref{t:ExtensionOC},
we have in fact shown that a distribution $f \in \D^\prime$ belongs to $\OC^\prime(\D, E_{\mathcal{W}})$ if and only if there is some $n \in \N$ and $f_\alpha \in E_\W$ for $|\alpha| \leq n$ such that
	\[ f = \sum_{|\alpha| \leq n} \partial^\alpha f_\alpha . \]
In \cite[Chapitre VII, Th\'eor\`eme IX, p.~244]{S-ThDist}, Schwartz showed this statement for the particular case of $\S'$, i.e., when $\W = \W_{\text{pol}}$.
Moreover, he also showed \cite[Chapitre VI, Th\'eor\`eme XXV, p.~201]{S-ThDist} that $\OC^\prime(\D, \D_{L^p}) = \D_{L^p}^\prime := (\D_{L^q})^\prime$, $p \in [1, \infty)$ and $q$ its conjugate index, that $\OC^\prime(\D, \mathcal{B}) = \B^\prime := (D_{L^1})^\prime$, and that $\OC^\prime(\D, \dot{\B}) = \dot{\B}^\prime := \overline{\D}^{\B^\prime}$.
We also refer to \cite{D-V-TopPropConvSpSTFT} for similar results for more general spaces.
\end{remark}

\subsection{The multiplier spaces}
\label{s5.2}
Let $\V, \W$ be weight function systems.
Recall that
\[
\B_\V = \bigcap_{n \in \N} (L^\infty)^n_{v_n}, \quad
\dot{\B}_\V = \bigcap_{n \in \N} (L^0)^n_{v_n}.
\]
We consider the spaces
\[ 
\OM(\dot{\B}_\V, E^\infty_\W) := \left\{ f \in \D^{\prime} \,:\, f \cdot \varphi \in E^\infty_\W \text{ for all } \varphi \in \dot{\B}_\V \right\}, 
\]
and
\[ 
\OM(\B_\V, E^\infty_\W) := \left\{ f \in \D^{\prime} \,:\, f \cdot \varphi \in E^\infty_\W \text{ for all } \varphi \in \B_\V \right\}.
\]

We find the following structural theorem for $\OM(\dot{\B}_\V, E^\infty_\W)$.

\begin{theorem}\label{t:CharOM}
Suppose the weight system $\V$ satisfies $(\condwM)$. 
Then $f \in \D^\prime$ belongs to $\OM(\dot{\B}_\V, E^\infty_\W)$ if and only if
$f \in C^\infty(\mathbb{R}^d)$ 
and for every 
$\alpha \in \N^d$ and $n \in \N$,  there exists $m \in \N$ such that
\begin{equation}\label{eq:OMCond}
f^{(\alpha)} \cdot \frac{w_n}{v_m} \in E_{\semloc}.
\end{equation}
\end{theorem}

\begin{proof}
Using Lemma~\ref{l:SmoothWeightFuncSyst}, we may assume that $\V$ consists of continuous functions.

$\bullet$ Suppose $f \in C^\infty(\mathbb{R}^d)$ and that for some $n \in \N$ there exists an integer $m \geq n$ for which $f^{(\beta)} w_n / v_m \in E_{\semloc}$ for all $|\beta| \leq n$.
For any $\varphi \in (L^0)^n_{v_m}$ and $\alpha \in \N^d$ with $|\alpha| \leq n$
\[
|(f \cdot \varphi)^{(\alpha)} w_n| \leq \sum_{\beta \leq \alpha} {\alpha \choose \beta} \left| f^{(\beta)} \cdot \varphi^{(\alpha - \beta)} w_n \right|
=\sum_{\beta \leq \alpha} {\alpha \choose \beta} \left| f^{(\beta)} \cdot \frac{w_n}{v_m} \cdot (\varphi^{(\alpha - \beta)} v_m) \right|
\]
by the Leibniz rule. Since $\varphi^{(\gamma)} v_m \in C_0(\mathbb{R}^d)$
for any $|\gamma| \le n$, it follows from Proposition~\ref{p:E=MEiffSemiLocal} and the solidity of $E$ that $(f \cdot \varphi)^{(\alpha)} w_n \in E$. This shows the sufficiency of the condition.

$\bullet$ Conversely, suppose $f \in \OM(\dot{\B}_\V, E^\infty_\W)$. 

For any compact set $K \Subset \mathbb{R}^d$ with non-empty interior, take $\phi_K \in \D$ such that $\phi_K \equiv 1$ on $K$. Then $f \cdot \phi_K \in E^\infty_\W  \subseteq C^\infty(\mathbb{R}^d)$, so $f \in C^\infty(\mathbb{R}^d)$.

Next, fix $\alpha \in \N^d$. By the closed graph theorem, the linear map
\[ 
\dot{\B}_\V \to E_\W : \quad \phi \mapsto f^{(\alpha)} \cdot \phi
\]
is continuous. Thus, for every $n \in \N$, there exist $m_0=m \in \N$ and $C > 0$ such that
\begin{equation}\label{eq:cdotfineq}
\| f^{(\alpha)} \cdot \phi \|_{E_{w_n}} \leq C \| \phi \|_{(L^\infty)^m_{v_m}} \quad \text{for all } \phi \in \dot{\B}_\V.
\end{equation}
By Lemma~\ref{l:StronglyReduced}, \eqref{eq:cdotfineq} extends to all $\phi \in (L^0)^{m+d}_{v_m}$.

Now fix $\rho \in C_0(\mathbb{R}^d)$. Define
\[ 
\rho^1(x) = \max_{|y| \leq 1} |\rho(x + y)|, \qquad \rho^2(x) = \max_{|y| \leq 1} |\rho^1(x + y)|.
\]
Then $\rho^1, \rho^2 \in C_0(\mathbb{R}^d)$. Choose
integers $m_2 \geq m_1 \geq m_0$ and $C > 0$ such that
\[ 
v_{m_j}(x + y) \leq C v_{m_{j+1}}(x), \quad \text{for all } x \in \mathbb{R}^d, ~ |y| \leq 1, ~ j = 0, 1.
\]
Let $\chi \in \D(\overline{B}(0,1))$ with $0 \leq \chi \leq 1$ and $\int \chi = 1$. Define
\[ 
\rho_m = \chi * \left( \frac{\rho^1}{v_{m_1}} \right).
\]
Then, for $|\beta| \leq m + d$,
\[
|\rho_m^{(\beta)}(x)| \leq \int_{\overline{B}(0,1)} |\chi^{(\beta)}(y)| \frac{\rho^1(x - y)}{v_{m_1}(x - y)} dy \leq C \|\chi\|_{(L^1)^{m + d}_1} \cdot \frac{\rho^2(x)}{v_m(x)},
\]
so $\rho_m \in (L^0)^{m+d}_{v_m}$, and \eqref{eq:cdotfineq}, extended to $(L^0)^{m+d}_{v_m}$, implies that $[f^{(\alpha)} \cdot \rho_m] w_n \in E$.

Also,
\[
\frac{|\rho(x)|}{v_{m_2}(x)} \leq C \int_{B(0,1)} \chi(y) \cdot \frac{\rho^1(x - y)}{v_{m_1}(x - y)} dy = C \rho_m(x).
\]
Hence,
\[
|f^{(\alpha)}| \cdot \frac{w_n}{v_{m_2}} \cdot |\rho| \leq C [|f^{(\alpha)}| \cdot \rho_m] w_n \in E.
\]
By the characteriation of $E_{\semloc}$ in Proposition~\ref{p:E=MEiffSemiLocal}, this proves \eqref{eq:OMCond}.
\end{proof}

\begin{corollary}
	\label{c:MultSpDotToInfty}
	Let $\V, \W$ be weight function systems such that $\V$ satisfies $(\condwM)$.
	Then $\OM(\dot{\B}_\V, E^\infty_\W) = \OM(\B_\V, (E_{\semloc})^\infty_\W) = \OM(\dot{\B}_\V, (E_{\semloc})^\infty_\W)$.
\end{corollary}

\begin{proof}
By Lemma \ref{l:DotandSemilocProp}$(ii)$, we have $(E_{\semloc})_{\semloc} = E_{\semloc}$.  
Since the condition \eqref{eq:OMCond} is expressed entirely in terms of $E_{\semloc}$, Theorem \ref{t:CharOM} immediately implies
\[
\OM(\dot{\B}_\V, E^\infty_\W) = \OM(\dot{\B}_\V, (E_{\semloc})^\infty_\W).
\]  
The inclusion $\dot{\B}_\V \subseteq \B_\V$ then yields
\[
\OM(\B_\V, (E_{\semloc})^\infty_\W) \subseteq \OM(\dot{\B}_\V, (E_{\semloc})^\infty_\W).
\]  

Conversely, let $f \in \OM(\dot{\B}_\V, (E_{\semloc})^\infty_\W)$ be arbitrary. Applying Theorem \ref{t:CharOM} again, we deduce that $f$ satisfies \eqref{eq:OMCond}.  
Replacing $\varphi \in (L^0)_{w_m}$ by $\varphi \in (L^\infty)_{w_m}$, an argument analogous to the proof of sufficiency in Theorem \ref{t:CharOM} shows that $f \in \OM(\B_\V, (E_{\semloc})^\infty_\W)$.
\end{proof}

Under an additional assumption on $\V$ or $\W$, the space $\OM(\B_\V, E^\infty_\W)$ coincides with three spaces
appearing in Corollary \ref{c:MultSpDotToInfty}.
\begin{corollary}
	\label{c:MultSpEquivDotInfty}
	Let $\V, \W$ be weight function systems satisfying $(\condwM)$.
	If either $\V$ or $\W$ satisfies $(\condwN)$, then $\OM(\B_\V, E^\infty_\W) = \OM(\dot{\B}_\V, E^\infty_\W)$.
\end{corollary}

\begin{proof}
	If $\V$ satisfies $(\condwN)$, then $\B_\V=(L^\infty)^\infty_\V=(L^0)^\infty_\V = \dot{\B}_\V$ by Lemma \ref{l:MEBasicTIBF}$(ii)$ and Proposition \ref{p:IndependenceSemiLoc}, and the identity follows directly.
	Otherwise, if $\W$ satisfies $(\condwN)$, then $(E_{\semloc})^\infty_\W = E^\infty_\W$ by Proposition \ref{p:IndependenceSemiLoc}, so that the identity is a consequence of Corollary \ref{c:MultSpDotToInfty}.
\end{proof}
We specialize by taking $\V = \W$.  
We conclude this section by characterizing when 
\[
\OM(\dot{\B}_\W, E^\infty_\W)
\] 
coincides exactly with the space of functions $f$ whose derivatives of all orders are dominated by some element of $\W$.  
This formulation recovers, and is closer in spirit to, the characterization that Schwartz gave for the multiplier space of $\S$ \cite{S-ThDist}.  
	
\begin{theorem}
\label{t:CharOMSingleW}
Let $\W$ be a weight function system satisfying $(\condwM)$.
Assume that
for every natural number $n$, there exists an $m \geq n$ such that
\begin{equation}\label{eq:SquarableWeight}
    \frac{w_n^2}{w_m} \in L^\infty.
\end{equation}
Then
	\begin{align}
		\label{eq:OMCondSameWeight}
		\OM(\dot{\B}_\W, E^\infty_\W)=
		\left\{ f \in C^\infty(\mathbb{R}^d) \,:\, \forall \alpha \in \N^d ~ \exists n \in \N \text{ such that } \frac{f^{(\alpha)}}{w_n} \in E_{\semloc} \right\}. 	
	\end{align}
In the case where $L^\infty \subseteq E_{\semloc}$, the converse is also true,
that is \eqref{eq:OMCondSameWeight} implies \eqref{eq:SquarableWeight}.
\end{theorem}

\begin{proof}
	Let us first show that \eqref{eq:SquarableWeight} suffices for \eqref{eq:OMCondSameWeight} to hold.
	For any $n \in \N$, we let $m_n \in \N$ be such that $w^2_n / w_{m_n} \in L^\infty$.
	Suppose $f \in C^\infty(\mathbb{R}^d)$ belongs to the right-hand side of \eqref{eq:OMCondSameWeight}.
	Fix any $\alpha \in \N^d$ and let $n_{\alpha} \in \N$ be such that $f^{(\alpha)} / w_{n_{\alpha}} \in E_{\semloc}$.
	Then, for any $n \geq n_\alpha$, since
		\[ |f^{(\alpha)}| \frac{w_n}{w_{m_n}} \leq \frac{w_n^2}{w_{m_n}} \cdot |f^{(\alpha)}| \frac{1}{w_{n_{\alpha}}} , \]
	it follows from the solidity of $E_{\semloc}$ that $f^{(\alpha)} w_n / w_{m_n} \in E_{\semloc}$.
	By Theorem \ref{t:CharOM} we then see that $f \in \OM(\dot{\B}_\W, E^\infty_\W)$.
	Conversely, if $f \in \OM(\dot{\B}_\W, E^\infty_\W)$, it follows from Theorem \ref{t:CharOM} that for any $\alpha \in \N^d$ there exists some $m \in \N$ such that $f^{(\alpha)} w_0 / w_m \in E_{\semloc}$. Hence, as $w_0 \geq 1$, by the solidity of $E_{\semloc}$, we have $f^{(\alpha)} / w_m \in E_{\semloc}$.
Thus, $f$ belongs to the right-hand side of \eqref{eq:OMCondSameWeight}.
	
	Suppose now that $L^\infty \subseteq E_{\semloc}$ and that \eqref{eq:OMCondSameWeight} is true.
	Since
$\B_\W \subseteq (E_{\semloc})^\infty_\W$,
 the inclusion $\OM(\dot{\B}_\W, \B_\W) \subseteq \OM(\dot{\B}_\W, (E_{\semloc})^\infty_\W)$ holds.
	However, using Proposition \ref{p:LowerUpperInclusions} and Corollary \ref{c:MultSpDotToInfty}, we deduce
		\[ \OM(\dot{\B}_\W, E^\infty_\W) = \OM(\dot{\B}_\W, (E_{\semloc})^\infty_\W) \subseteq \OM(\dot{\B}_\W, \B_\W) . \]
	Therefore, $\OM(\dot{\B}_\W, \B_\W) = \OM(\dot{\B}_\W, E^\infty_\W)$, so that we may assume without loss of generality that $E = L^\infty$.
	
	By Lemma \ref{l:SmoothWeightFuncSyst} we may assume that each $w_n$
itself is smooth and that there exists an integer $m \geq n$ for which $w_n^{(\alpha)} / w_m\in L^\infty$ for all $\alpha \in \N^d$. 
	Then each $w_n$ is an element of the right-hand side of \eqref{eq:OMCondSameWeight}, so that they are also elements of $\OM(\dot{\B}_\W, \B_\W)$.
	Therefore, \eqref{eq:SquarableWeight} follows from Theorem \ref{t:CharOM}
	with $f=w_n$ and $\alpha=0$.
\end{proof}

\begin{remark}
%
$(i)$
Concerning $\D_E=E^\infty_{\bf 1}$,
one has $\OM(\dot{\B}, \D_E) = \OM(\B, \D_{E_{\semloc}}) = \D_{E_{\semloc}}$
due to Corollary \ref{c:MultSpDotToInfty} and Theorem \ref{t:CharOMSingleW}.

$(ii)$
Assume
that $1 / w_N \in L^0$ for some $N \in \N$
and that \eqref{eq:SquarableWeight} holds. For any $n \geq N$, we let $m \geq n$ be an integer such that $w_n^2 / w_m \in L^\infty$.
Then
	\[
0 \le \frac{w_n}{w_m} \le \frac{w_n^2}{w_m} \cdot \frac{1}{w_N} \in L^0 , \qquad n \geq N . \]
As a result, $\W$ satisfies $(\condwN)$.
Therefore, in this case, by Corollary \ref{c:MultSpEquivDotInfty}, $\OM(\B_\W, E^\infty_\W) = \OM(\dot{\B}_\W, E^\infty_\W)$.
\end{remark}

\section{Characterizations via mollification}
\label{s6}

In this final section, we consider how to detect whether a distribution belongs to any of the spaces $E^\infty_\W$, $\OC^\prime(\D, E^\infty_\W)$, or $\OM(\dot{\B}_\V, E^\infty_\W)$, by looking at its mollifications.
When dealing with a single window, such characterizations for $E^\infty_\W$ are often possible.

\begin{proposition}\label{p:250408-1}
Assume that $(\condwM)$ is satisfied.
	Let $\phi \in \D$ with $\int \phi \neq 0$.
	Suppose $E$ is a closed subspace of $E_{\semloc}$.
	For any $f \in \D^\prime$ it holds: $f \in E^\infty_\W$ if and only if $\{f * \phi_j \,:\, j \in \N\}$ is bounded in $E^\infty_\W$.
\end{proposition}

\begin{proof}
	If $f \in E^\infty_\W$, then $\{f * \phi_j \,:\, j \in \N\}$ is bounded in $E^\infty_\W$ by Lemma \ref{l:ConvGenOmega}.
	We now show the converse.

	As $E^\infty_\W \subseteq C^\infty(\R^d)$ continuously, the set $B = \{f * \phi_j \,:\, j \in \N\}$ is also bounded in $C^\infty(\R^d)$.
	As the latter space is Montel, the set $B$ is relatively compact there. 
	Consequently, there exists a subsequence $(j_k)_{k \in \N}$ of $\N$ such that $f * \phi_{j_k}$ converges to some $g \in C^\infty(\R^d)$.
	But, by \eqref{eq:250421-12}, we must necessarily have that $g = (\int \phi) f$. 
	Consequently, $f \in C^\infty(\R^d)$. \\
	\indent
	We first consider the maximal case: $E = E_{\semloc}$.
    	We will show that $f w_n \in E_{\semloc}$ for any $n \in \N$.
    	Since $\{ f^{(\alpha)} * \phi_j \,:\, j \in \N\}$ is also bounded in $(E_{\semloc})^\infty_\W$ for any $\alpha \in \N^d$, this would complete the proof for $E = E_{\semloc}$.
    	Put $M = \sup_{j \in \N} \| f * \phi_j \|_{(E_{\semloc})_{w_n}} < \infty$.
    	Fix some $K \Subset \R^d$ and let $\chi \in \D$ be such that $\chi \equiv 1$ on $K$ and for which $0 \leq \chi \leq 1$.
	By Lemma \ref{l:SmoothWeightFuncSyst}, we may assume each $w_n$ is continuous.
	Then $\chi \cdot [f * \phi_j] \cdot w_n \to \chi \cdot f \cdot w_n$ in $C_{\rm c}(\R^d)$, for any $n \in \N$, hence also in $E$.
	Consequently, using the solidity of $E$, we have that for any $n \in \N$,
		\[ \| 1_{K} f \|_{E_{w_n}} \leq \| \chi \cdot f \cdot w_n \|_E = \lim_{j \to \infty} \| \chi \cdot [f * \phi_j] \cdot w_n \|_E \leq \sup_{j \in \N} \| f * \phi_j \|_{(E_{\semloc})_{w_n}} \leq M . \]
	As $K$ was chosen arbitrarily, we conclude that $f \in E_{\semloc}$.
	
	Suppose now $E$ is a strict closed subspace of $E_{\semloc}$.
	Note that by our previous argument, we already have $f  \in (E_{\semloc})^\infty_\W$.
	Take any $\varphi \in C_{\rm c}(\R^d)$ and $\alpha \in \N^d$.
	Then, $f^{(\alpha)} * \phi_j * \varphi \in E^\infty_\W$ for any $j \in \N$ by Proposition \ref{p:EWConvCompact} and our assumptions.
	Since $\phi_j * \varphi \to \varphi$ in $\D$, another application of Proposition \ref{p:EWConvCompact} gives that $f^{(\alpha)} * \phi_j * \varphi \to f^{(\alpha)} * \varphi$ in $(E_{\semloc})^\infty_\W$.
	In particular, as $E$ is closed in $E_{\semloc}$, it follows that $f^{(\alpha)} * \varphi \in E^\infty_\W$.
	By \eqref{eq:parametrix_formula}, as $\varphi$ and $\alpha$ were chosen arbitrarily, we may conclude $f \in E^\infty_\W$.
\end{proof}

%
%

The question of when $E$ is closed in $E_{\semloc}$ is considered in \cite{B-F-BanachSpDistTwoMod}.
In particular, it is true when $E = E_{\semloc}$ or, by Proposition \ref{p:E=MEiffSemiLocal}, when $E = E_{\approxim}$.
Consequently, the characterization in Proposition \ref{p:250408-1} holds if
	\begin{itemize}
		\item[$\bullet$] $E = E_{\approxim}$ or $E = E_{\semloc}$. In particular, if $E$ has an absolutely continuous norm or satisfies the Fatou property, see Lemma \ref{l:ECoincidesWithDerivSp};
		\item[$\bullet$] $\W$ satisfies $(\condwN)$, by Proposition~\ref{p:IndependenceSemiLoc}.
	\end{itemize}
In case where $E$ is not a closed subspace of $E_{\semloc}$, we do not know whether the characterization
as in Propsoition \ref{p:250408-1} still holds.

The situation becomes significantly more flexible when mollifying with at least two windows.
As we will see in Theorem \ref{t:KWDecomp}, by utilizing two well-chosen functions $\chi^0$ and $\chi^1$, it can already suffice for
	\[ \sup_{|\alpha| \leq n} \| \partial^\alpha [f * \chi_j^\ell] \|_E = O(2^{rj}) , \qquad n \in \N, ~ \ell = 0, 1 , \]
for some $r > 0$, to conclude $f$ is in $E^\infty_\W$.
Essentially, the coefficient $r$ allows us to consider mollifications with derivatives of the windows $\chi^\ell$.
This amounts to the ability to cancel a finite number of moments of the window, that is, to dampen out low frequencies of the distribution.
Note that in \eqref{eq:250421-12}, we would have convergence to zero, hinting at the need for two windows.

%

We now proceed to make everything concrete.
We will start by introducing specific pairs of windows that will allow our results as mentioned above.

\subsection{Moment-wise decomposition factorization property}

We consider the functional equation for functions $\varphi, \psi \in \mathcal{D}$:
\begin{equation}\label{eq:DecompTestFuncCond}
2^d \varphi(2 \cdot) - \varphi = \Delta \psi.
\end{equation}
It is known that nontrivial solutions exist, as established by Schott \cite{S-FuncSpExpWeighI}.  
Any pair $(\varphi, \psi) \in \D \times \D$ satisfying \eqref{eq:DecompTestFuncCond} can be used to decompose distributions into smooth functions.

\begin{lemma}
	\label{l:D'Decomp}
	Let $\varphi, \psi \in \D \setminus \{0\}$ be such that \eqref{eq:DecompTestFuncCond} holds.
Assume further
\begin{equation}\label{eq:250408-71}
\int_{{\mathbb R}^d}\varphi(x){\rm d}x=1.
\end{equation}
 Then, for any $f \in \D^\prime$,
		\begin{equation}
			\label{eq:D'Decomp} 
			f = f * \varphi + \sum_{j \in \N} 4^{-j} [(\Delta f) * \psi_j]  . 
		\end{equation}
\end{lemma}

\begin{proof}
	For any $j \in \N$, we have
		\[ 2^{(j + 1)d} \varphi(2^{j + 1} \cdot \, ) - 2^{j d} \varphi(2^j \cdot \, ) = 2^{j(d - 2)} \Delta (\psi(2^j \cdot \, )) . \]
Assuming (\ref{eq:250408-71}) holds,
we see that
		\[ \delta = \varphi + \sum_{j \in \N}( 2^{(j + 1)d} \varphi(2^{j + 1} \cdot \, ) - 2^{j d} \varphi(2^j \cdot \, ) )= \varphi + \sum_{j \in \N} 4^{-j} [2^{jd} \Delta(\psi(2^j \cdot \, ))] , \]
	where the sum converges in $\D^\prime$. As $f = f * \delta$ for any $f \in \D^\prime$, the result follows.
\end{proof}

We now introduce the following factorization property for pairs of test functions.

\begin{definition}
\label{def:MDFP}
We say that $(\chi^0, \chi^1) \in \D \times \D$ has the \emph{moment-wise decomposition factorization property} 
$($MDFP$)$ if for some $K \Subset \mathbb{R}^d$ there exist
non-trivial functions $(\varphi_L^\ell, \psi_L^\ell) \in \mathcal{D}(K) \times \mathcal{D}_L(K)$, 
$\ell \in \{0, 1\}$ and  $L \in \N$, such that
$\varphi^0_L, \psi^0_L$ satisfy
\begin{equation}\label{eq:DecompTestFuncCond2}
2^d \varphi^0_L(2 \cdot) - \varphi^0_L = \Delta \psi^0_L,
\end{equation}
\begin{equation}\label{eq:250408-712}
\int_{{\mathbb R}^d}\varphi^0_L(x){\rm d}x=1 ,
\end{equation}
and
\begin{equation}\label{eq:MDFPFact}
\varphi^0_L = \chi^0 * \varphi^0_L + \chi^1 * \varphi^1_L, \quad \psi^0_L = \chi^0 * \psi^0_L + \chi^1 * \psi^1_L.
\end{equation}
\end{definition}

Note that if $(\chi^0, \chi^1) \in \D \times \D$ satisfies the MDFP, then \eqref{eq:IntroDecomp} follows immediately for every $L \in \N$ from \eqref{eq:D'Decomp} and \eqref{eq:MDFPFact} by taking $\varphi^\ell = \varphi^\ell_L$ and $\psi^\ell = \psi^\ell_L$ for $\ell \in \{0, 1\}$.

The main goal of this subsection is to show that there do exist $(\chi^0, \chi^1)$ satisfying the MDFP.
\begin{proposition}\label{p:ExistMDPF}
There exist $(\chi^0, \chi^1) \in \mathcal{D} \times \mathcal{D}$ satisfying the MDFP.
\end{proposition}

To establish Proposition \ref{p:ExistMDPF}, we provide a sufficient condition for the MDFP. This is achieved through the theory of ultradifferentiable functions in the sense of Braun-Meise-Taylor \cite{B-M-T-UltradiffFuncFourierAnal}. We begin by recalling several key definitions and function spaces.

\begin{definition}
    A \emph{Braun-Meise-Taylor weight function} (abbreviated to \emph{BMT-weight function}) \cite{B-M-T-UltradiffFuncFourierAnal} is a continuous, non-decreasing function $\omega : [0, \infty) \to [0, \infty)$ with $\omega(0) = 0$ satisfying the following conditions:
    \begin{itemize}
        \item[$(\alpha)$] There exists a constant $C > 0$ such that $\omega(2t) \leq C \omega(t)$ for all $t \geq 0$.
        \item[$(\beta)$] The integral $\int_1^\infty \frac{\omega(t)}{1 + t^2} dt$ converges.
        \item[$(\delta)$] The function $\varphi : [0, \infty) \to [0, \infty)$ defined by $\varphi(t) := \omega(e^t)$,
$t \geq 0$, is convex.
    \end{itemize}
We write $\varphi^*$
for    the \emph{Young conjugate} of $\varphi$
and
it is given by
    \[
        \varphi^*(x) := \sup_{y \geq 0} \big( xy - \varphi(y) \big) \quad (x \ge0).
    \]
The function $\varphi^*$ is convex and increasing, $(\varphi^*)^* = \varphi$, and the function $y \mapsto \varphi^*(y) / y$ is increasing on $[0, \infty)$ and tends to infinity as $y \to \infty$.
\end{definition}

\begin{definition}
    Let $\omega$ be a BMT-weight function. For any $h > 0$ and any compact set $K \Subset \mathbb{R}^d$, we define the Banach space
    \[
        \mathcal{D}^{\omega, h}(K) := \left\{ f \in \mathcal{D}(K) \,:\, \sup_{\alpha \in \mathbb{N}^d} \| f^{(\alpha)} \|_{L^{\infty}} e^{- \frac{1}{h} \varphi^*(h |\alpha|)} < \infty \right\}.
    \]
    The associated Fr\'echet space is given by
    \[
        \mathcal{D}^{(\omega)}(K) := \bigcap_{h > 0} \mathcal{D}^{\omega, h}(K).
    \]
    If $K$ has a non-empty interior, then $\mathcal{D}^{(\omega)}(K)$ is non-trivial \cite[Remark 3.2(1)]{B-M-T-UltradiffFuncFourierAnal}. For any $L \in \mathbb{N}$, we further define
    \[
        \mathcal{D}^{(\omega)}_L(K) := \mathcal{D}^{(\omega)}(K) \cap \mathcal{D}_L(K).
    \]
\end{definition}

\begin{definition}{\rm \cite{G-AlmostPeriodicUltradistBeurlingRoumieu}}
    Let $\omega$ be a BMT-weight function. An entire function $G$ on $\mathbb{C}^d$ satisfying \[ \log |G(z)| ={\rm O}(\omega(|z|)) \quad \text{as } |z| \to \infty \] defines, for any compact set $K \Subset \mathbb{R}^d$, a continuous linear operator
    \[
        G(D) : \mathcal{D}^{(\omega)}(K) \to \mathcal{D}^{(\omega)}(K), \quad f \mapsto \sum_{\alpha \in \mathbb{N}^d} (-i)^{|\alpha|} \frac{G^{(\alpha)}(0)}{\alpha!} f^{(\alpha)}.
    \]
    We refer to this as an \emph{ultradifferentiable operator of class $(\omega)$}.
\end{definition}

We present the following sufficient condition for the MDFP.
Remark that the existence of such $\chi_0$ and $\chi_1$
is guaranteed
by
\cite[Corollary 2.6]{G-AlmostPeriodicUltradistBeurlingRoumieu}.
\begin{proposition}
    \label{p:SuffCondDFP}
    Let $\chi_0, \chi_1 \in \mathcal{D}$, let $\omega$ be a BMT-weight function, and let $G$ be an ultradifferentiable operator of class $(\omega)$. Assume that
    \[ \chi_0 + G(D) \chi_1 = \delta. \]
    Then the pair $(\chi_0, \chi_1)$ satisfies the MDFP.
\end{proposition}

\begin{proof}
    Let $K \Subset \mathbb{R}^d$ and $\phi \in \D^{(\omega)}(K)$. Then
    \[
        \phi = \delta * \phi = (\chi_0 + G(D) \chi_1) * \phi = \chi_0 * \phi + \chi_1 * [G(D) \phi].
    \]
    Now, $G(D)\phi \in \D^{(\omega)}(K)$. Moreover, if $\phi \in \D^{(\omega)}_L(K)$, then $G(D)\phi \in \D^{(\omega)}_L(K)$ as well.
    It therefore suffices to show that, for some compact $K \Subset \mathbb{R}^d$ and for each $L \in \N$, there exists a non-trivial pair $(\varphi_L, \psi_L) \in \D^{(\omega)}(K) \times \D^{(\omega)}_L(K)$ satisfying
    \[ 2^d \varphi_L(2 \, \cdot \,) - \varphi_L = \Delta \psi_L , \qquad \int_{\mathbb{R}^d} \varphi_L(x){\rm d}x=1. \]
    Then
    \[ (\varphi^0_L, \psi^0_L) := (\varphi_L, \psi_L), \quad (\varphi^1_L, \psi^1_L) := (G(D) \varphi_L, G(D) \psi_L) \]
    will satisfy the conditions of the MDFP with respect to $(\chi^0, \chi^1)$.

    We follow the proof of \cite[Proposition 4.1]{S-FuncSpExpWeighI}, referring also to \cite[Theorem 1.40, p.~86]{S-ThBesovSp}.
    We will construct sequences $\rho_L \in \D^{(\omega)}(\overline{B}(0, 1))$ and $\zeta_L \in \D^{(\omega)}(\overline{B}(0,2))$, $L \in \N$, such that $\int \rho_L = 1$ and
    \begin{equation} \label{eq:FuncEqAlt}
        \Delta^{L+1} \zeta_L = \rho_L - 2^{-d} \rho_L(2^{-1} \, \cdot \,) .
    \end{equation}
    Then we may define
    \[ \varphi_L := 2^{-d} \rho_L(2^{-1} \, \cdot \,), \quad \psi_L := \Delta^L \zeta_L , \]
   which would complete the proof.

    Begin with any $\phi_0 \in \D^{(\omega)}([1/2,1]) \setminus \{0\}$. Let $|\mathbb{S}^{d-1}|$ denote the surface area of the unit sphere in $\mathbb{R}^d$. We may assume
    \begin{equation}\label{eq:250408-82}
        |\mathbb{S}^{d-1}| \int_0^1 r^{d-1}\phi_0(r)\,dr = 1.
    \end{equation}

    Define recursively
    \[ \phi_{L+1} = \mu_L \phi_L + \lambda_L \phi_L(2 \, \cdot \,), \]
    where the coefficients satisfy
    \begin{align}
        \mu_L + 2^{-d} \lambda_L &= 1, \label{eq:250408-81} \\
\nonumber
        \mu_L + 2^{-d - 2(L+1)} \lambda_L &= 0.
    \end{align}
    Then $\phi_L \in \D^{(\omega)}([2^{-(L+1)},1]) \setminus \{0\}$ for all $L \in \N$, and \eqref{eq:250408-81} together with \eqref{eq:250408-82} inductively gives
    \begin{equation}\label{eq:250408-83}
        \int_{\mathbb{R}^d} \phi_L(|x|)\,dx = 1.
    \end{equation}

    Set $\eta_L := \phi_L - 2^{-d} \phi_L(2^{-1} \, \cdot \,) \in \D^{(\omega)}([2^{-(L+1)}, 2])$.
    Introduce the linear operator $T : C[0, \infty) \to C[0, \infty)$ given by
    \[
        (Tf)(r) := \int_0^r \left( \int_0^t \left(\frac{s}{t}\right)^{d-1} f(s) ds \right) dt=\int_0^r\left( \int_0^1 s^{d-1} f(t s) ds \right) dt.
    \]
    Then
    \[
        (Tf)^{(n)}(r) = \int_0^1 s^{d+n-2} f^{(n-1)}(rs) ds , \qquad n \geq 1.
    \]

    Assume $f$ satisfies
    \begin{equation}\label{eq:250408-111}
        M_h = \sup_{n \in \N, r \in [0,2]} |f^{(n)}(r)| e^{-\frac{1}{h} \varphi^*(hn)} < \infty, \quad \text{for all } h > 0.
    \end{equation}
    We find
    \[
        |(Tf)^{(n)}(r)| e^{-\frac{1}{h} \varphi^*(hn)} \leq M_h e^{\frac{1}{h}(\varphi^*(h(n-1)) - \varphi^*(hn))} \leq M_h , \qquad r \in [0, 2] .
    \]

    Let $T^L$ be the $L$-fold composition of $T$. If $f$ satisfies \eqref{eq:250408-111}, then we deduce, by induction on $L$,
    \[
        \sup_{n \in \N, x \in [0,2]} |(T^L f)^{(n)}(x)| e^{-\frac{1}{h} \varphi^*(hn)} < \infty
    \]
for all $h>0$.
    Applying this to $T^{L+1} \eta_L$ yields
    \[
        \sup_{\alpha \in \N^d, r \in [0, 2]} |(T^{L+1} \eta_L)^{(n)}(r)| e^{-\frac{1}{h} \varphi^*(h n)} < \infty.
    \]

    By \cite{S-FuncSpExpWeighI} and \cite[p.~87]{S-ThBesovSp}, $T^{L+1} \eta_L$ coincides with an even polynomial $P_L$ of degree $2L$ for $r \geq 2$. Define
    \[
        \zeta_L(x) := (T^{L+1} \eta_L)(|x|) - P_L(|x|).
    \]
    It follows from \cite[Proposition 8.4.1, p.~281]{H-AnalLinPDOI} that $\zeta_L \in \D^{(\omega)}(\overline{B}(0,2))$.
    
    For any $f \in C^2([0, \infty))$ vanishing near the origin, we have
    \[
    	\Delta[(Tf)(|x|)] = (Tf)^{\prime\prime}(|x|) + \frac{d-1}{|x|} (Tf)^\prime(|x|) = f(|x|)
    \]
    Then, by iteration, it follows that
    \[
        \Delta^{L+1} \zeta_L(x) = \eta_L(|x|) = \phi_L(|x|) - 2^{-d} \phi_L(2^{-1}|x|).
    \]
    Similar to before, the function $\rho_L = \phi_L(|\cdot|)$ belongs to $\D^{(\omega)}(\overline{B}(0, 2))$ and satisfies $\int \rho_L = 1$ by \eqref{eq:250408-83}.
    Our proof is complete.
\end{proof}

\begin{proof}[Proof of Proposition \ref{p:ExistMDPF}]
Let $\omega$ be any BMT-weight function. By \cite[Corollary 2.6]{G-AlmostPeriodicUltradistBeurlingRoumieu}, there exists an ultradifferentiable operator $G(D)$ of class $(\omega)$, along with functions $\chi_0 \in \D^{(\omega)}(K)$ and $\chi_1 \in \D^{\omega, 1}(K)$, for some $K \Subset \mathbb{R}^d$, such that
\[ \delta = \chi_0 + G(D) \chi_1. \]
Thus, the pair $(\chi_0, \chi_1)$ satisfies the MDFP condition by Proposition \ref{p:SuffCondDFP}.
\end{proof}

\subsection{Gelfand-Shilov-type spaces}
We establish the following characterization for spaces of the type $E^\infty_\W$.
The existence of a pair $(\chi^0, \chi^1) \in \mathcal{D} \times \mathcal{D}$ satisfying the MDFP is guaranteed by Proposition~\ref{p:ExistMDPF}.
\begin{theorem}
\label{t:KWDecomp}
Let $\W$ be a weight function system satisfying $(\condwM)$.
Suppose $(\chi^0, \chi^1) \in \mathcal{D} \times \mathcal{D}$ has the MDFP.
Then for any $f \in \mathcal{D}'$, the following conditions are equivalent{\rm:}
\begin{itemize}
  \item[$(i)$] $f \in E^\infty_\W${\rm;}
  \item[$(ii)$] For every $\phi \in \mathcal{D}$, the set $\{ f * \phi_j : j \in \mathbb{N} \}$ is bounded in $E^\infty_\W${\rm;}
  \item[$(iii)$] For every $\phi \in \mathcal{D}$, there exists $r > 0$ such that the set $\{ 2^{-rj} (f * \phi_j) : j \in \mathbb{N} \}$ is bounded in $E^\infty_\W${\rm;}
  \item[$(iv)$] There exists $r > 0$ such that the set $\{ 2^{-rj} (f * \chi_j^\ell) : j \in \mathbb{N},\ \ell = 0, 1 \}$ is bounded in $E^\infty_\W$.
\end{itemize}
\end{theorem}
We present several technical lemmas required for the proof of Theorem~\ref{t:KWDecomp}. 
These lemmas are formulated and proved in a more general setting than strictly necessary for the current theorem, 
since they will be employed again for later proofs.

\begin{lemma}
	\label{l:ContinuousMapPhi}
	Let $K \Subset \mathbb{R}^d$, and $C>0$.
	Suppose that $\omega, \nu : \mathbb{R}^d \to (0, \infty)$ are as in Lemma \ref{l:ConvGenOmega}.
	Then for any $\phi \in \D(K)$, $n \in \N$, and bounded subset $B \subset E^n_\nu$, the set
	\[
		\{ f * \phi_j \,:\, f \in B,\, j \in \N \}
	\]
	is bounded in $E^n_\omega$.
\end{lemma}

\begin{proof}
	Indeed, by Lemma \ref{l:ConvGenOmega}, for any $\phi \in \D(K)$, $n \in \N$, and $f \in E^n_\nu$, we have
		\[
			\| \partial^\alpha [ f * \phi_j] \omega \|_E
			\leq C\|\phi\|_{L^1} \|f^{(\alpha)} \nu \|_E , \qquad |\alpha| \leq n .
		\]
Hence, the set in question is bounded.
\end{proof}

\begin{lemma}
\label{l:PhiReconstruct}
Let $\omega : \mathbb{R}^d \to (0, \infty)$ be a locally integrable function.
Suppose that $\varphi, \psi \in \D \setminus \{0\}$ satisfy conditions \eqref{eq:DecompTestFuncCond} and \eqref{eq:250408-71}.
Let $f \in \D'$ be such that $f * \varphi \in E^n_\omega$ and the sequence 
$\{ f * \psi_j \,:\, j \in \N \}$ is bounded in $E^{n + 2}_\omega$, for some $n \in \N$.
Then $f \in E^n_\omega$.
\end{lemma}

\begin{proof}
	 For any $|\alpha| \leq n$,
		\begin{multline*}
			\| \partial^\alpha [f * \varphi] \omega \|_{E} + \sum_{j \in \N} 4^{-j} \| \partial^{\alpha} [\Delta f * \psi_j] \omega \|_{E} \\
			\leq \| \partial^\alpha [f * \varphi] \omega \|_{E} + \frac{4}{3} \sup_{j \in \N} \max_{|\alpha| \leq n + 2} \| \partial^\alpha [f * \psi_j] \omega \|_E
			< \infty.
		\end{multline*}
	Then $f \in E^{n}_{\omega}$ by Lemma \ref{l:D'Decomp}.
\end{proof}

The next lemma considers the behavior of functions under mollification with windows having vanishing moments.
\begin{lemma}
	\label{l:KWDecompVanishMoments}
Let $\omega, \nu, \kappa : \mathbb{R}^d \to (0, \infty)$ be locally integrable weight functions, and assume that $\nu$ is smooth.

Suppose there exist a symmetric, convex compact set $K \Subset \mathbb{R}^d$, $L \in \mathbb{N}$, and constants $C, C_\alpha > 0$ such that
\begin{equation}\label{eq:250411-1}
	\omega(x + y) \leq C \nu(x)
\end{equation}
and
\begin{equation}\label{eq:250411-2}
	\left| \partial^{\alpha} \nu(x) \right| \leq C_\alpha \kappa(x)
\end{equation}
for all $x \in \mathbb{R}^d$, $y \in K$, and $\alpha \in \N^d$ with $|\alpha| \leq 2L + d + 2$.
Let
\[
d' :=
\begin{cases}
\frac{d + 1}{2}, & \text{if } d \text{ is odd}, \\
\frac{d}{2} + 1, & \text{if } d \text{ is even}.
\end{cases}
\]

Then
\begin{equation}
	\label{eq:L1VDecomp}
B':=
	\left\{
		2^{2Lj} f * \phi_j \mid f \in B,\; j \in \mathbb{N}
	\right\} \subset E^n_{\omega}
\end{equation}
is bounded in $E^n_{\omega}$
for every $\phi \in \mathcal{D}_{2L}(K)$,
every $n \in \mathbb{N}$, and every bounded subset $B \subset E^{n + 2L + 2d'}_{\kappa}$.
\end{lemma}

\begin{proof}
Before proceeding to the proof,
 we first derive an estimate based on the assumption that $\phi \in \mathcal{D}_{2L}(K)$: 
by applying the Taylor expansion and using the fact that $\phi \in \mathcal{D}_{2L}$, we obtain, for some constant $C_\phi > 0$ independent of $j$, the inequality
\begin{equation}\label{eq:2502505-1}
|\mathcal{F}\{\check{\phi}_j\}(\xi)| \leq C_\phi\, 2^{-2Lj} |\xi|^{2L} \leq C_\phi\, 2^{-2Lj} \langle \xi\rangle^{2L}, \qquad \xi \in \mathbb{R}^d.
\end{equation}

First fix $n \in {\mathbb N}$ and $\alpha \in {\mathbb N}^d$ with $|\alpha| \le n$.
Also fix a bounded subset \( B \subset E^{n + 2(L + d')}_{\kappa} \), and let \( \phi \in \mathcal{D}_{2L}(K) \). 
Choose \( \chi \in \mathcal{D} \) such that \( \chi \equiv 1 \) on $K$.

Let $f \in B$.
Note that for any \( j \in \mathbb{N} \), we have
\[
| \partial^{\alpha} [f * \phi_j] \, \omega | 
\leq C 
| [f^{(\alpha)} \nu] * \phi_j |
\]
due to (\ref{eq:250411-1}).
By (\ref{eq:250411-2}),
\[
\left| \partial^{\beta} [f^{(\alpha)} \nu] \right| \leq \sum_{\gamma \leq \beta} \binom{\beta}{\gamma} \left| f^{(\alpha + \beta - \gamma)} \nu^{(\gamma)} \right| \leq  \sum_{\gamma \leq \beta} C_\gamma\binom{\beta}{\gamma} \left| f^{(\alpha + \beta - \gamma)} \right| \kappa
\]
for any \( f \in B \), \( |\alpha| \leq n \), and \( |\beta| \leq 2(L + d') \).
Since $E$ is solid and $B$ is a bounded subset of $E^{n + 2(L + d')}_{\kappa}$,
\[
\left\{ \partial^{\beta} [f^{(\alpha)} \nu] \,:\, f \in B,\ |\alpha| \leq n,\ |\beta| \leq 2(L + d') \right\}
\]
is bounded in \( E \). 

Now, for certain constants \( c_{\beta, \gamma} \), with 
$\beta, \gamma$ satisfying \( |\beta + \gamma| \leq 2(L + d') \), we have the identity
\[
(1-\Delta_y)^{L + d'} \left( \chi(y) T_y [f^{(\alpha)} \nu] \right)
= \sum_{|\beta + \gamma| \leq 2(L + d')} c_{\beta, \gamma} \chi^{(\gamma)}(y) T_y\{ \partial^{\beta}[f^{(\alpha)} \nu] \}.
\]
This yields two consequences:
\begin{enumerate}
    \item By the solidity and the translation invariance of \( E \), the set
    \[
    \left\{ (1-\Delta_y)^{L + d'} \left( \chi(y)\, T_y [f^{(\alpha)} \nu] \right) \,:\, f \in B,\ |\alpha| \leq n,\ y \in \mathbb{R}^d \right\}
    \]
    is bounded in \( E \).

    \item Since \( E \subset L^1_{\mathrm{loc}}(\mathbb{R}^d) \), for any \( f \in B \), \( |\alpha| \leq n \), and \( x \in \mathbb{R}^d \), the map
    \[
    y \mapsto (1-\Delta_y)^{L + d'} \left( \chi(y)\, f^{(\alpha)}(x - y)\, \nu(x - y) \right)
    \]
    belongs to \( L^1 \).
\end{enumerate}
Now, $\phi_j = \phi_j \chi$ for any $j \in \N$, so that
\[\int_{\mathbb{R}^d} \phi_j(y) [f^{(\alpha)}(x - y) \nu(x - y)] dy
			= \int_{\mathbb{R}^d} \phi_j(y) \chi(y)[f^{(\alpha)}(x - y) \nu(x - y)] dy.
\]
The Plancherel theorem gives
		\begin{align*}
&\int_{\mathbb{R}^d} \phi_j(y) [f^{(\alpha)}(x - y) \nu(x - y)] dy \\
&\qquad= \int_{\mathbb{R}^d} \mathcal{F}\{\check{\phi}_j\}(\xi) \mathcal{F}\{ \chi(\cdot) [f^{(\alpha)}(x - \cdot) \nu(x - \cdot)] \} (\xi) d\xi\\
			&\qquad = \int_{\mathbb{R}^d} \langle\xi\rangle^{-2(L + d')} \mathcal{F}\{\check{\phi}_j\}(\xi) \cdot \langle\xi\rangle^{2(L + d')} \mathcal{F}\{ \chi(\cdot) [f^{(\alpha)}(x - \cdot) \nu(x - \cdot)] \} (\xi) d\xi .
\end{align*}
By another application of the Plancherel theorem, we find that for any $f \in B$, $|\alpha| \leq n$, $j \in \N$, and $x \in \mathbb{R}^d$
		\begin{align*}
			&\int_{\mathbb{R}^d} \phi_j(y) [f^{(\alpha)}(x - y) \nu(x - y)] dy \\
			&\qquad = \int_{\mathbb{R}^d} \langle\xi\rangle^{-2(L + d')} \mathcal{F}\{\check{\phi}_j\}(\xi) \mathcal{F}\left\{ (1-\Delta)^{L + d'} \left( \chi(\cdot) [f^{(\alpha)}(x - \cdot) \nu(x - \cdot)] \right) \right\} (\xi) d\xi \\
			&\qquad = \int_{\mathbb{R}^d} \mathcal{F}\{\langle\,\cdot\,\rangle^{-2(L + d')} \mathcal{F}\{\check{\phi}_j\} \}(y) \cdot (1-\Delta_y)^{L + d'} \left( \chi(y) [f^{(\alpha)}(x - y) \nu(x - y)] \right) dy ,
		\end{align*}
	so that we thus have
		\[ [f^{(\alpha)} \nu] * \phi_j = \int_{\supp \chi} \mathcal{F}\{\langle\,\cdot\,\rangle^{-2(L + d')} \mathcal{F}\{\check{\phi}_j\} \}(y) \cdot (1-\Delta_y)^{L + d'} \left( \chi(y) T_y [f^{(\alpha)} \nu] \right) dy , \]
	where the integral on the right-hand side may be regarded as a Bochner integral in $E$. 
	Putting things together, we find, for some $C' > 0$ depending on $B$, the estimate
		\[ \| \partial^\alpha [f * \phi_j] \omega \|_E \leq C' \| \langle\,\cdot\,\rangle^{-2(L + d')} \mathcal{F}\{\check{\phi}_j\} \|_{L^1} , \qquad f \in B, ~ j \in \N, ~ |\alpha| \leq n . \]
From (\ref{eq:2502505-1}), we conclude
		\[ 2^{2Lj} \| \partial^\alpha [f * \phi_j] \omega \|_E \leq C' C_{\phi} \|\langle\,\cdot\,\rangle^{-2d'}\|_{L^1}  , \qquad f \in B, ~ j \in \N, ~ |\alpha| \leq n . \]
In view of the definition of $d'$, the quantity
$\|\langle\,\cdot\,\rangle^{-2d'}\|_{L^1}$ is finite.
Thus $B'$ is bounded in $E^n_\omega$.
\end{proof}

\begin{proof}[Proof of Theorem \ref{t:KWDecomp}]
	$(i) \Rightarrow (ii)$: By Lemma \ref{l:ContinuousMapPhi}.
	
	$(ii) \Rightarrow (iii)$: Direct, with $r = 0$.
	
	$(iii) \Rightarrow (iv)$: Simply choose $\phi = \chi^\ell$ for $\ell \in \{0, 1\}$.
	
	$(iv) \Rightarrow (i)$: Suppose $r > 0$ is such that $\{ 2^{-rj} [f * \chi^\ell_j] \,:\, j \in \N, \ell = 0, 1 \}$ is bounded in $E^\infty_\W$.
	Without loss of generality, we may suppose that $r$ is an even integer.
	Let $K \Subset \mathbb{R}^d$ be symmetric and convex such that there exist non-trivial $(\varphi^\ell_L, \psi^\ell_L)\in \D(K) \times \D_L(K)$, for $\ell \in \{0, 1\}$ and $L \in \N$, where $\varphi^0_L, \psi^0_L$ satisfy \eqref{eq:DecompTestFuncCond2}, \eqref{eq:250408-712}, and \eqref{eq:MDFPFact} for each $L \in \N$.
	We suppose that $2L>r$.
	First, by our assumptions, we have that $f * \chi^\ell = f * \chi^\ell_0 \in E^\infty_\W$ for $\ell = 0, 1$, so that by Proposition \ref{p:EWConvCompact}
		\[ f * \varphi^0_{r} = [f * \chi^0] * \varphi^0_{r} + [f * \chi^1] * \varphi^1_{r} \in E^\infty_\W . \]
	Second, for any $n \in \N$
		\[
			|f * (\psi^0_{r})_j \cdot w_n| 
			= \left| \sum_{\ell = 0}^1 [f * \chi^\ell_j] * (\psi^\ell_{r})_j \cdot w_n \right|
			\leq C \sum_{\ell = 0}^1 2^{j r} \left| (2^{-rj} [f * \chi^\ell_j]) * (\psi^\ell_{r})_j \cdot w_n \right| .
		\]
	Using Lemma \ref{l:SmoothWeightFuncSyst}, we may assume that $\W$ consists  of smooth functions, and that for every $n \in \N$ there exist $n \leq m \leq k$ and $C > 0$ such that $w_n(x + y) \leq C w_m(x)$ and $|w^{(\alpha)}_m(x)| \leq C w_k(x)$ for any $x \in \mathbb{R}^d$, $y \in K$, and $|\alpha| \leq r + d + 4$.
	Since $\{ 2^{-rj} [f * \chi^\ell_j] \,:\, j \in \N, \ell = 0, 1 \}$ is a bounded set in $E^{n + r + d + 4}_{w_k}$, by applying Lemma \ref{l:KWDecompVanishMoments} with $\omega = w_n$, $\nu = w_m$, and $\kappa = w_k$, and using the solidity of $E$, we may conclude that $\{ f * (\psi^0_{r})_j \,:\, j \in \N \}$ is a bounded subset of $E^{n + 2}_{w_n}$.
	In view of Lemma \ref{l:PhiReconstruct}, with $\varphi = \varphi^0_{r}$ and $\psi = \psi^0_{r}$, we see that $f \in E^n_{w_n}$.
Since $n \in \N$ is arbitrary, we finally conclude that $f \in E^\infty_\W$.
\end{proof}

We now characterize the convolutor space $\OC'(\D, E_\W)$ and the multiplier space $\OM(\dot{\B}_\V, E^\infty_\W)$. We begin with the former.
\begin{theorem}
\label{t:OCDecomp}
Let $\W$ be a weight system satisfying condition~$(\text{wM})$.
Suppose that $(\chi^0, \chi^1) \in \D \times \D$ has MDFP. Then, for any $f \in \D'$, the following are equivalent{\rm:}
\begin{itemize}
  \item[$(i)$] $f \in \OC'(\D, E_\W)${\rm;}
  \item[$(ii)$] For every $\phi \in \D$ and $n \in \N$, there exists $r > 0$ such that the set 
$$B_1:=\{ 2^{-rj} [f * \phi_j] w_n : j \in \N \}$$ is bounded in $E${\rm;}
  \item[$(iii)$]
For every $n \in \N$, there exists $r > 0$ such that the set 
$$B_2:=\{ 2^{-rj} [f * \chi^\ell_j] w_n : j \in \N,\ \ell = 0, 1 \}$$ is bounded in $E$.
\end{itemize}
\end{theorem}

\begin{proof}[Proof of Theorem~\ref{t:OCDecomp}]
(i) $\Rightarrow$ (ii): Fix a compact set $K \Subset \mathbb{R}^d$ and $n \in \N$. Then there exist $r \in \N$ and $C > 0$ such that
\[
  \| [f * \phi] w_n \|_E \leq C \max_{|\alpha| \leq r} \| \phi^{(\alpha)} \|_{L^1}, \qquad \phi \in \D(K),
\]
by the closed graph theorem. Consequently, for any $\phi \in \D(K)$ and $j \in \N$,
\[
  \| [f * \phi_j] w_n \|_E
  \leq C \max_{|\alpha| \leq r} \int_{\mathbb{R}^d} |\partial^\alpha (2^{jd} \phi(2^j x))| \, dx
  \leq C 2^{jr} \max_{|\alpha| \leq r} \| \phi^{(\alpha)} \|_{L^1},
\]
implying that $B_1$ is bounded in $E$.

(ii) $\Rightarrow$ (iii): This follows immediately by taking $\phi = \chi^\ell$ for $\ell \in \{0,1\}$.

(iii) $\Rightarrow$ (i): We aim to show that $f * \phi \in E_\W$ for all $\phi \in \D$.

Let $K \Subset \mathbb{R}^d$ be such that, for each $L \in \N$ and $\ell \in \{0, 1\}$, there exists a nontrivial pair
$(\varphi^\ell_L, \psi^\ell_L) \in \D(K) \times \D_L(K)$, satisfying \eqref{eq:DecompTestFuncCond2}, \eqref{eq:250408-712}, and \eqref{eq:MDFPFact}.

Fix $\phi \in \D$ and let $\widetilde{K}$ be a symmetric, convex compact set containing $K + \supp \phi$, the Minkowski sum of $K$ and $\supp \phi$.
Take $n \in \N$ and choose $m \geq n$ and $C > 0$ such that
\[
  w_n(x + y) \leq C w_m(x), \qquad \text{for all } x \in \mathbb{R}^d,\ y \in \widetilde{K}.
\]
Let $r_n$ be an even integer such that the set $\{ 2^{-r_nj} [f * \chi^\ell_j] w_m : j \in \N,\ \ell = 0, 1 \}$ is bounded in $E$.
Then, for all $|\alpha| \leq n$ and $j \in \N$,
\[
  |\partial^\alpha ([f * \phi] * (\psi^0_{r_n})_j) w_n|
  \leq C \sum_{\ell = 0}^{1} \left( 2^{-r_nj} |f * \chi^\ell_j| w_m \right) * \left( 2^{jr_n} |\partial^\alpha [\phi * (\psi^\ell_{r_n})_j]| \right).
\]
Since $r_n$ is an even integer, we may apply
Lemma~\ref{l:KWDecompVanishMoments}
 with $L=\frac{r_n}{2}$, $\omega = \nu = \kappa = 1$, and $E = L^1$.
 Thus, the set
\[
  \{ 2^{jr_n} \partial^\alpha [\phi * (\psi^\ell_{r_n})] : j \in \N,\ |\alpha| \leq n,\ \ell = 0,1 \}
\]
is bounded in $L^1$. Therefore, using the solidity of $E$, it follows that
\[
  \{ [f * \phi] * (\psi^0_{r_n})_j : j \in \N \}
\]
is bounded in $E^n_{w_n}$ for every $n \in \N$. Similarly, one shows that $[f * \phi] * \varphi^0_{r_n} \in E^n_{w_n}$ for all $n \in \N$.
From Lemma \ref{l:PhiReconstruct}, $f * \phi \in E^\infty_\W$.
Since $\phi$ is arbitrary, we conclude that $f \in\OC'(\D, E_\W)$.
\end{proof}

We move on to the multiplier spaces.

\begin{theorem}
	\label{t:OMDecomp}
	Let $\V, \W$ be weight function systems satisfying $(\condwM)$.
	Suppose that we have $(\chi^0, \chi^1) \in \D \times \D$ has the MDFP.
	For any $f \in \D^\prime$, the following are equivalent:
	\begin{itemize}
		\item[$(i)$] $f \in \OM(\dot{\B}_\V, E^\infty_\W)${\rm;}
		\item[$(ii)$] For any $\phi \in \D$ and $n \in \N$, 
there exists an integer
$m \in \N$ such that $\{ f * \phi_j \,:\, j \in \N \}$ is bounded in $(E_{\semloc})^n_{w_n / v_m}$.
		\item[$(iii)$] For any $\phi \in \D$, there exists a constant $r > 0$ so that for every $n \in \N$ there is an integer $m \in \N$ such that $\{ 2^{-rj} [f * \phi_j] \,:\, j \in \N \}$ is bounded in $(E_{\semloc})^n_{w_n / v_m}$.
		\item[$(iv)$] There exists a $r > 0$, so that for every $n \in \N$ there is an integer $m \in \N$ such that $\{ 2^{-rj} [f * \chi^\ell_j] \,:\, j \in \N, \ell = 0, 1 \}$ is bounded in $(E_{\semloc})^n_{w_n / v_m}$.
	\end{itemize}
\end{theorem}

\begin{proof}
	$(i) \Longrightarrow (ii)$
	Fix $f \in \OM(\dot{\B}_\V, E^\infty_\W)$ and $\phi \in \D$.
	For any $n \in \N$,
using $(\condwM)$,  let $n' \geq n$ and $C > 0$ be such that $w_n(x + y) \leq C w_{n'}(x)$ for all $x \in \mathbb{R}^d$ and $y \in \supp \phi$.
	Then by Theorem \ref{t:CharOM}, there is some $m' \in \N$ for which $f^{(\alpha)} w_{n'} / v_{m'} \in E_{\semloc}$ for all $|\alpha| \leq n$.
	Finally, there is an integer $m \geq m'$ and $C' > 0$ so that $v_{m'}(x - y) \leq C' v_{m}(x)$ for all $x \in \mathbb{R}^d$ and $y \in \supp \phi$.
	In particular, we have that
		\[ \frac{w_n(x + y)}{v_m(x+y)} \leq C C' \frac{w_{n'}(x)}{v_{m'}(x)} , \qquad x \in \mathbb{R}^d, ~ y \in \supp \phi . \]
	Then $\{ f * \phi_j \,:\, j \in \N \}$ is bounded in $(E_{\semloc})^n_{w_n / v_m}$ by Lemma \ref{l:ContinuousMapPhi}.
	
	$(ii) \Longrightarrow (ii)$ Trival (with $r = 0$).
	
	$(iii) \Longrightarrow (iv)$ Direct, by choosing $\phi = \chi^\ell$ for $\ell \in \{0, 1\}$.
	
	$(iv) \Longrightarrow (i)$ 
	Without loss of generality, we may suppose $r$ is an even integer.
	Let $K \Subset \mathbb{R}^d$ be symmetric and convex, such that there exist non-trivial pairs $(\varphi^\ell_L, \psi^\ell_L)\in \D(K) \times \D_L(K)$, for $\ell \in \{0, 1\}$ and $L \in \N$, satisfying \eqref{eq:DecompTestFuncCond2}, \eqref{eq:250408-712}, and \eqref{eq:MDFPFact}.
	Fix some $n \in \N$ and let $n \leq m \leq k$ and $C > 0$ be such that $w_n(x + y) \leq C w_m(x)$ and $w_m(x + y) \leq C w_k(x)$ for any $x \in \mathbb{R}^d$ and $y \in 2K$.
	There is an integer $k' \in \N$ such that $\{ 2^{-rj} [f * \chi^\ell_j] \,:\, j \in \N, \ell = 0, 1 \}$ is bounded in $(E_{\semloc})^{n + r + d + 4}_{w_k / v_{k'}}$.
	Then there exist $k' \leq m' \leq n'$ and some $C' > 0$ such that $v_{k'}(x + y) \leq C v_{m'}(x)$ and $v_{m'}(x + y) \leq C v_{n'}(x)$ for any $x \in \mathbb{R}^d$ and $y \in 2K$.
	We thus have that
		\[ \frac{w_n(x + y)}{v_{n'}(x + y)} \leq C C' \frac{w_m(x)}{v_{m'}(x)} , \qquad \frac{w_m(x + y)}{v_{m'}(x + y)} \leq C C' \frac{w_k(x)}{v_{k'}(x)} , \qquad x \in \mathbb{R}^d, y \in 2K . \]
	Take any $\eta \in \D(K)$ which is non-negative and satisfies $\int \eta = 1$.
	By Lemma \ref{l:SandwhichSmoothFunc}, there exists some $C'' > 0$ such that, for any $x \in \mathbb{R}^d$, $y \in K$, and $|\alpha| \leq r + d + 2$,
		\[ \frac{w_n(x + y)}{v_{n'}(x + y)} \leq C'' \left(\frac{w_m}{v_{m'}}\right)_{\eta}(x)
 , \qquad \left| \left(\frac{w_m}{v_{m'}}\right)_{\eta}^{(\alpha)}(x)\right| \leq C'' \frac{w_k(x)}{v_{k'}(x)} . \]
	Applying Lemma \ref{l:KWDecompVanishMoments}, with $\omega = w_n / v_{n'}$, $\nu = (w_m / v_{m'})_\eta$, and $\kappa = w_k / v_{k'}$, shows us that $\{ [f * \chi_j^\ell] * (\psi_r^\ell)_j \,:\, j \in \N, \ell = 0, 1 \}$ is a bounded subset of $(E_{\semloc})^{n + 2}_{w_n / v_{n'}}$. 
	The remainder of the proof can now be done analogously as that of $(iv) \Longrightarrow (i)$ of Theorem \ref{t:KWDecomp}, where we show that $f * \varphi^0_{r} \in (E_{\semloc})^n_{w_n / v_{n'}}$ using Lemma \ref{l:ConvGenOmega}.
\end{proof}

We specialize again by taking $\V=\W$.
Let us consider a specific form of Theorem \ref{t:OMDecomp}, utilizing the structure of \( \OM(\dot{\mathcal{B}}_\mathcal{W}, E^\infty_\mathcal{W}) \) described in Theorem~\ref{t:CharOMSingleW}.

\begin{theorem}
  \label{t:OMDecompSingleWeight}
  Let $\W$ be a weight function system satisfying conditions \((\mathrm{wM})\) and~\eqref{eq:SquarableWeight}.
  Assume that \( (\chi^0, \chi^1) \in \mathcal{D} \times \mathcal{D} \) has the MDFP.
  Then, for any \( f \in \mathcal{D}' \), the following statements are equivalent{\rm:}
  \begin{itemize}
    \item[$(i)$] \( f \in \OM(\dot{\mathcal{B}}_\mathcal{W}, E^\infty_\mathcal{W}) \).

    \item[$(ii)$] For every \( \phi \in \mathcal{D} \) and \( \alpha \in \mathbb{N}^d \), there exists \( n \in \mathbb{N} \) such that the set
    \[
      \left\{ \frac{\partial^\alpha [f * \phi_j]}{w_n} \,:\, j \in \mathbb{N} \right\}
    \]
    is bounded in \( E_{\semloc} \).

    \item[$(iii)$] For every \( \phi \in \mathcal{D} \), there exists \( r > 0 \) such that for each \( \alpha \in \mathbb{N}^d \), there exists \( n \in \mathbb{N} \) for which the set
    \[
      \left\{ 2^{-rj} \frac{\partial^\alpha [f * \phi_j]}{w_n} \,:\, j \in \mathbb{N} \right\}
    \]
    is bounded in \( E_{\semloc} \).

    \item[$(iv)$] There exists \( r > 0 \) such that for every \( \alpha \in \mathbb{N}^d \), there exists \( n \in \mathbb{N} \) for which the set
    \[
      \left\{ 2^{-rj} \frac{\partial^\alpha [f * \chi_j^\ell]}{w_n} \,:\, j \in \mathbb{N},\ \ell = 0, 1 \right\}
    \]
    is bounded in \( E_{\semloc} \).
  \end{itemize}
\end{theorem}

\begin{proof}
Since $w_n \ge 1$,
 by condition~\eqref{eq:SquarableWeight}, each of the statements \((ii)\), \((iii)\), and \((iv)\) is equivalent to its corresponding formulations in Theorem~\ref{t:OMDecomp}\((ii)\) and \((iii)\)
with $\V=\W$.
\end{proof}

We conclude this paper by proving Theorem \ref{t:Intro} from the Introduction.

\begin{proof}[Proof of Theorem \ref{t:Intro}]
We have the weight function system $\W_\omega = (\exp[n \cdot \omega])_{n \in \N}$, which satisfies $(\condwM)$ by \eqref{eq:moderateness}.
Note that $\W_\omega$ automatically satisfies \eqref{eq:SquarableWeight}
with $m=2n$.
As we saw in Lemma \ref{l:MEBasicTIBF},
$(L^p)_{\semloc} = L^p$.
Thus $\OM(\dot{\B}_{\W_\omega}, \K^p_{\W_\omega}) = \OM(\K^\infty_\omega, \K^p_\omega)$ by Corollary \ref{c:MultSpDotToInfty}.
Let $E=L^p$.
The result then follows from Theorems \ref{t:KWDecomp}, \ref{t:OCDecomp},
and \ref{t:OMDecompSingleWeight} applied to $E^\infty_{\W_\omega}$, $\OC^\prime(\D, E_{\W_\omega})
=\OC^\prime(\D, E_{\W})$, and $\OM(\dot{\B}_{\W_\omega}, \K^p_{\W_\omega})$, respectively.
\end{proof}

\section*{Acknowledgement}

The authors are grateful to the anonymous reviewers for their careful reading, which improved the presentation of this paper.

\end{document}